\documentclass[10pt,twoside]{article}

\usepackage{amsmath,amsthm,amssymb}
\usepackage{color} 
\usepackage{mathtools,mathptmx}
\usepackage{indentfirst}
\usepackage{mathptmx}

\usepackage{amsbsy}  
\usepackage{amsfonts}
\usepackage{graphicx}
\usepackage{tkz-euclide}

\usepackage{array}
\usepackage{geometry}
\usepackage[bookmarks=true,colorlinks=true, pdfstartview=FitV, linkcolor=black, citecolor=blue, urlcolor=black]{hyperref}

\usepackage{cases}

\newtheorem{theorem}{Theorem}[section]
\newtheorem{proposition}[theorem]{Proposition}
\newtheorem{corollary}[theorem]{Corollary}
\newtheorem{lemma}[theorem]{Lemma}

\newtheorem{remark}[theorem]{Remark}

\newcommand{\dis}{\displaystyle}
\newcommand{\R}{\mathbb{R}}
\newcommand{\C}{\mathbb{C}}
\newcommand{\Z}{\mathbb{Z}}
\newcommand{\N}{\mathbb{N}}

\newcommand{\Real}{\operatorname{Re}}

\newcommand{\W}{\mathcal{W}}
\newcommand{\Wvec}{\overline{\W}}
\newcommand{\V}{\mathcal{V}}

\newcommand{\parameter}[1]{{[#1]}}
\newcommand{\roundparameter}[1]{{(#1)}}

\newcommand{\kk}{{\parameter{k}}}

\newcommand{\zero}{{\parameter{0}}}
\newcommand{\one}{{\parameter{1}}}
\newcommand{\two}{{\parameter{2}}}

\newcommand{\PP}{\mathcal{P}}

\newcommand{\BesselK}[2][x]{\mathrm{K}_{#2}\left(#1\right)}

\newcommand{\aaa}{\mathrm{a}}
\newcommand{\bbb}{\mathrm{b}}

\newcommand{\dd}{\mathrm{d}}
\newcommand{\dx}[1][x]{\mathrm{d}#1}
\newcommand{\DiffOp}[1][x]{\frac{\dd}{\dx[#1]}}
\newcommand{\DiffOpHigherOrder}[2][x]{\frac{\dd^{#2}}{\dx[#1]^{#2}}}
\newcommand{\DiffOpFunction}[2][x]{\frac{\dd #2}{\dx[#1]}}
\newcommand{\DiffOpHigherOrderFunction}[3][x]{\frac{\dd^{#2}#3}{\dx[#1]^{#2}}}
\newcommand{\pochhammer}[2][n]{\left(#2\right)_{#1}}

\newcommand{\Hypergeometric}[5][x]
{{}_{#2} F_{#3} \left(\begin{matrix} #4 \vspace*{0,1 cm}\\ #5 \end{matrix}
	\, ; \, #1\right)}

\newcommand{\twoFtwo}[3][x]{\Hypergeometric[#1]{2}{2}{#2}{#3}}

\newcommand{\U}{\textbf{U}}
\newcommand{\UCal}{\mathcal{U}}
\newcommand{\KummerU}[3][x]{\textbf{U}\left(#2,#3;#1\right)}

\newcommand\matrixtwobytwo[4]{\begin{bmatrix} \dis #1 & \dis #2 \vspace*{0,1 cm} \\ 
		\dis #3 & \dis #4 \end{bmatrix}}

\newcommand\twovector[2]
{\begin{bmatrix} \dis #1 \vspace*{0,1 cm} \\ \dis #2 \end{bmatrix}}

\newcommand{\UCalvectorK}[2][x]
{\twovector{\UCal_0^{#2}\left(#1\right)}{\UCal_1^{#2}\left(#1\right)}}

\newcommand{\UCalvector}[1][x]
{\twovector{\UCal_0\left(#1\right)}{\UCal_1\left(#1\right)}}

\newcommand{\floor}[1]{\left\lfloor #1 \right\rfloor}
\newcommand{\UCalvec}{\overline{\UCal}}

\newcommand{\polyseq}[1][P_n(x)]{ \{#1 \}_{n\in\N}}  

\newcommand{\KK}{\huge \textbf{K} \normalsize}
\newcommand{\contfrac}[3][j=1]{\mathop{\KK}\limits_{#1}^{\infty}\left(\frac{#2}{#3}\right)}

\numberwithin{equation}{section}
\numberwithin{theorem}{section}

\newcommand{\n}{\vec{n}}

\newcommand{\indexedvector}[2][n]{(#1_0,\cdots,#1_{#2-1})}

\newcommand{\e}{\mathrm{e}}

\newcommand{\diag}{\mathrm{diag}}
\newcommand{\Hn}{\mathrm{H}_n}
\newcommand{\moment}[1][k]{m_{#1}}

\usepackage{fancyhdr}
\begin{document}
	
\title{Multiple orthogonal polynomials associated with confluent hypergeometric functions}
\author{H\'elder Lima\footnote{Address: School of Mathematics, Statistics and Actuarial Sciences, University of Kent, Sibson Building, Parkwood Road, Canterbury, CT2 7FS,  UK \newline 
Email addresses: (H\'elder Lima) has27@kent.ac.uk and (Ana Loureiro) a.loureiro@kent.ac.uk }\  \ and Ana Loureiro$^{*}$}

\date{\today}
\maketitle
	
\noindent\uppercase{Abstract.}
We introduce and analyse a new family of multiple orthogonal polynomials of hypergeometric type with respect to two measures supported on the positive real line which can be described in terms of confluent hypergeometric functions of the second kind. These two measures form a Nikishin system. 
Our focus is on the multiple orthogonal polynomials for indices on the step line. 
The sequences of the derivatives of both type I and type II polynomials with respect to these indices are again multiple orthogonal and  they correspond to the original sequences with shifted parameters.
For the type I polynomials, we provide a Rodrigues formula. We characterise the type II polynomials via their explicit expression as a terminating generalised hypergeometric series, as solutions to a third-order differential equation and via their recurrence relation. The latter involves recurrence coefficients which are unbounded and asymptotically periodic. Based on this information we deduce the asymptotic behaviour of the largest zeros of the type II polynomials.
We also discuss limiting relations between these polynomials and the multiple orthogonal polynomials with respect to the modified Bessel weights. Particular choices on the parameters for the type II polynomials under discussion correspond to the cubic components of the already known threefold symmetric Hahn-classical multiple orthogonal polynomials on star-like sets.

\noindent\textbf{Keywords:} 
\textit{Multiple orthogonal polynomials, confluent hypergeometric function, Nikishin system, Rodrigues formula, generalised hypergeometric series, differential equation, recurrence relation, Hahn-classical}

\noindent \textbf{Mathematics Subject Classification 2000:} Primary: 33C45, 42C05, Secondary: 33C10, 33C15, 33C20

\section{Introduction and motivation}

The main aim of this paper is to investigate the {\it multiple orthogonal polynomials} with respect to two absolutely continuous measures supported on the positive real line and admitting an integral representation via weight functions $\W(x;a,b,c)$ and $\W(x;a,b,c+1)$ where 
\begin{align}
\label{modified Tricomi weight definition}
\W(x;a,b,c)=\frac{\Gamma(c+b+1)}{\Gamma(a+1)\Gamma(b+1)}\,\e^{-x}x^a\,\KummerU{c}{a-b+1}, 
\end{align}
with $a$, $b$, $c$ such that $a>-1$, $b>-1$ and $c>\max\{0,a-b\}$. The weight functions involve the {\it confluent hypergeometric function of the second kind} $\dis\KummerU[x]{\alpha}{\beta}$, also known as the {\it Tricomi function}, which is a solution of the second-order differential equation 
(see \cite[Eq.~13.2.1]{DLMF}) 
\begin{equation}
\label{Kummer equation}
x\,\DiffOpHigherOrderFunction{2}{y}+(\beta-x)\DiffOpFunction{y}-\alpha y=0,
\end{equation}
and, provided that $\Real(\alpha)>0$ and $\dis|\arg(x)|<\frac{\pi}{2}$, it admits the integral representation (see \cite[Eq.~13.4.4]{DLMF})
\begin{align*}
\KummerU{\alpha}{\beta}=\frac{1}{\Gamma(\alpha)}
\int_{0}^{\infty}t^{\alpha-1}(t+1)^{\beta-\alpha-1}e^{-tx}\dx[t].
\end{align*}
The conditions on the parameters $a>-1$, $b>-1$ and $c>\max\{0,a-b\}$ guarantee the convergence of  the integral of the modified Tricomi weight on the positive real line, precisely  (see \cite[Eq.~13.10.7]{DLMF}):
\begin{equation}
\int_{0}^{\infty}\e^{-x}x^a\,\KummerU{c}{a+1-b}\dx
=\frac{\Gamma(a+1)\Gamma(b+1)}{\Gamma(c+b+1)}.
\end{equation}
Therefore the weight function $\dis\W(x;a,b,c)$ is a probability density function whose moments are  given by
\begin{align}
\label{moments}
\moment(a,b,c)
=\int_{0}^{\infty}x^k\ \W(x;a,b,c)\, \dx
=\frac{\pochhammer[k]{a+1}\pochhammer[k]{b+1}}{\pochhammer[k]{c+b+1}},\quad k\in \N,
\end{align}
where, as usual, $\pochhammer[k]{z}$ denotes the Pochhammer symbol defined by 
\[
	\dis\pochhammer[0]{z}=1\quad \text{and} \quad \dis\pochhammer[k]{z}:=z(z+1)\cdots(z+k-1) ,\quad k\in\N\backslash\{0\}.
\] 
Here and throughout the text, $\N=\Z_0^+=\{0,1,2,\cdots\}$. 
When referring to $\polyseq$ as a polynomial sequence it is assumed that $P_n$ is a polynomial of a single variable with degree exactly $n$ and we consistently deal with monic polynomials, unless stated otherwise.

Research on multiple orthogonal polynomials has received a focus of attention in the past decennia, partly motivated by their applicability to different areas of mathematics and mathematical physics. In particular, they have been utilised in the description of rational solutions to Painlev\'e equations \cite{ClarkManf} as well as in random matrix theory. For instance, the investigation of singular values of products of Ginibre matrices uses multiple orthogonal polynomials associated with weight functions expressed in terms of Meijer G-functions \cite{KuijlaarsStiv14}. 
If only two measures are involved, then those Meijer G-functions are hypergeometric or confluent hypergeometric functions. 
This research offers a thorough investigation of a collection of multiple orthogonal polynomials that fits within this category.

{\it Multiple orthogonal polynomials} are a generalisation of (standard) orthogonal polynomials. 
Their orthogonality measures are spread across a vector of $r\in\Z^+$ measures and they are polynomials on a single variable depending on a multi-index $\n=(n_0,\cdots,n_{r-1})\in\N^r$ of length $|\n|=n_0+\cdots+n_{r-1}$. 
There are two types of multiple orthogonal polynomials with respect to a system of $r$ measures $(\mu_0,\cdots,\mu_{r-1})$. 

The \textit{type I multiple orthogonal polynomials} for $\n=(n_0,\cdots,n_{r-1})\in\N^r$ are given by a vector $\left(A_{\n}^{(0)},\cdots,A_{\n}^{(r-1)}\right)$ of $r$ polynomials, with $\deg A_{\n}^{(j)}\leq n_j-1$, for each $0\leq j\leq r-1$, satisfying the orthogonality and normalisation conditions
\begin{align}
\label{orthogonality conditions type I}
\sum_{j=0}^{r-1}\int x^kA_{\n}^{(j)}(x)\dd\mu_j(x)=
\begin{cases}
0, &\text{ if } 0\leq k\leq|\n|-2,\\
1, &\text{ if } k=|\n|-1.
\end{cases}
\end{align}

If the measures $\mu_j(x)$ are absolutely continuous with respect to a common positive measure $\mu$, that is, if we can write $\dis\dd\mu_j(x)=w_j(x)\dd\mu(x)$, for each $0\leq j\leq r-1$ and for some weight functions $w_j(x)$, then the type I function is
\begin{align}
\label{type I function definition}
Q_{\n}(x)=\sum_{j=0}^{r-1}A_{\n}^{(j)}(x)\ w_j(x)
\end{align}
and the conditions in \eqref{orthogonality conditions type I} become
\begin{align}
\label{orthogonality conditions type I function}
\int x^kQ_{\n}(x)\dd\mu(x)
=\begin{cases}
0, &\text{ if } 0\leq k\leq|\n|-2,\\
1, &\text{ if } k=|\n|-1.
\end{cases}
\end{align}
In the case of $r=2$ measures, we use the notation $A_{\n}$ for $A_{\n}^{(0)}$ and $B_{\n}$ for $A_{\n}^{(1)}$.

The \textit{type II multiple orthogonal polynomial} for $\n=(n_0,\cdots,n_{r-1})\in\N^r$ consists of monic polynomials $P_{\n}$ of degree $|\n|$ which satisfies, for each $0\leq j\leq r-1$, the orthogonality conditions
\begin{align}
\label{orthogonality conditions type II}
\int x^kP_{\n}(x)\dd\mu_j(x)=0,
\;\;\;0\leq k\leq n_j-1.
\end{align}
	
For both types of multiple orthogonality, the case where the number of measures is $r=1$ corresponds to standard orthogonality.
A polynomial sequence $\polyseq$ is orthogonal with respect to a measure $\mu$ if
\begin{align}
\label{standard orthogonality conditions}
\int x^kP_n(x)\dd\mu(x)
=\begin{cases}
0, &\text{ if } 0\leq k\leq n-1,\\
N_n\neq 0, &\text{ if } n=k.
\end{cases}
\end{align}

The orthogonality conditions for type I and type II multiple orthogonal polynomials give a non-homogeneous system of $|\n|$ linear equations for the $|\n|$ unknown coefficients of the vector of polynomials $\left(A_{\n}^{(0)},\cdots,A_{\n}^{(r-1)}\right)$ in \eqref{orthogonality conditions type I} or the polynomials $P_{\n}(x)$ in \eqref{orthogonality conditions type II}. 
If the solution exists, it is unique and the corresponding matrices of the system for type I and type II are the transpose to each other. 
However it is possible that this system doesn't have a solution, unless further conditions are imposed
(unlike standard orthogonality on the real line, the existence of such solutions is not a trivial matter). 
If there is a unique solution, then the multi-index $\n$ is called \textit{normal} and if all multi-indices are normal, the system is a \textit{perfect system}. 

An example of systems known to be perfect are the Algebraic Tchebyshev systems, or simply \textit{AT-systems}.
A vector of measures $\dis(\mu_0,\cdots,\mu_{r-1})$ is an AT-system on an interval $\Delta$ for a multi-index $\n=(n_0,\cdots,n_{r-1})\in\N^r$ if the measures $\mu_j(x)$ are absolutely continuous with respect to a common positive measure $\mu$ on $\Delta$, that is, $\dis\dd\mu_j(x)=w_j(x)\dd\mu(x)$, for each $0\leq j\leq r-1$ and for some weight functions $w_j(x)$, and the set of functions
\begin{equation*}
\bigcup_{j=0}^{r-1}\left\{w_j(x),xw_j(x),\cdots,x^{n_j-1}w_j(x)\right\}
\end{equation*}
forms a Chebyshev system on $\Delta$, that is, if for any polynomials $p_0,\cdots,p_{r-1}$ of degree not greater than $n_j-1$, for each $0\leq j\leq r-1$, and not all equal to $0$, the function $\dis\sum_{j=0}^{r-1}p_j(x)w_j(x)$ has at most $|\n|-1$ zeros on $\Delta$.
A vector of measures $\dis(\mu_0,\cdots,\mu_{r-1})$ is an AT-system on an interval $\Delta$ if it is an AT-system on $\Delta$ for every multi-index in $\N^r$.

Another special example of a perfect system is a \textit{Nikishin system} (first introduced in \cite{NikishinSystems}). 
We say that two measures  $(\mu_0,\mu_1)$  form a Nikishin system of order $2$, if they are both supported on an interval $\Delta_0$ and if there exists a positive measure $\sigma$ on an interval $\Delta_1$ with $\Delta_0\cap\Delta_1=\emptyset$ such that
\begin{align}
\label{Nikishin system - ratio of the measures}
\frac{\dd\mu_1(x)}{\dd\mu_0(x)}=\int_{\Delta_1}\frac{\dd\sigma(t)}{x-t}.
\end{align}

The definition of a Nikishin system can be generalised to define a Nikishin system of $r>2$ measures. 
It was proved in \cite{NikishinSystemsArePerfect} that every Nikishin system is perfect (see also \cite{NikishinSystemsArePerfectCaseOfUnboundedAndTouchingSupports} for the cases where the supports of the measures are unbounded or where consecutive intervals touch at one point).
More precisely, it is proved in \cite{NikishinSystemsArePerfect} and \cite{NikishinSystemsArePerfectCaseOfUnboundedAndTouchingSupports} that every Nikishin system is an AT-system, therefore it is perfect. 
Moreover, for every AT-system and for any $\n\in\N^r$, the type I function for $Q_{\n}$ defined by \eqref{type I function definition} has exactly $|\n|-1$ sign changes on $\Delta$ and the type II multiple orthogonal polynomial $P_{\n}$ has $|\n|$ simple zeros on $\Delta$ which satisfy an interlacing property as there is always a zero of $P_{\n}$ between two consecutive zeros of $P_{\n+\vec{e}_k}$, for each $0\leq k\leq r-1$, where $\vec{e}_k\in\N^r$ is the multi-index that has all entries equal to $0$ except the entry of index $k$ which is equal to $1$.
As a Nikishin system is always an AT-system, the same properties hold for Nikishin systems.

The main contribution of this paper is on multi-indices on the step line. 
A multi-index $\indexedvector{r}\in\N^r$ is on the step-line if $n_0\geq n_1\geq\cdots\geq n_{r-1}\geq n_0-1$ or, equivalently, if there exists $m\in\N$ and $0\leq j\leq r-1$ such that
\begin{align*}
n_k=
\begin{cases}
m+1, & \text{ if } 0\leq k<j,\\
  m, & \text{ if } j\leq k\leq r-1.
\end{cases}
\end{align*}

For any $r\in\Z^+$ and for each $n\in\N$, there is a unique multi-index of length $n$ on the step line of $\N^r$.
More precisely, if $n=rm+j$, with $m,j\in\N$ and $0\leq j\leq r-1$, the multi-index of length $n$ is $\dis\n=\indexedvector{r}\in\N^r$ with entries as described above.
Hence, when the number of measures is fixed and we only consider multi-indices on the step line, we can replace the multi-index of the multiple orthogonal polynomials of both type I and type II by its length without any ambiguity.
When $r=2$, the indexes on the step line are illustrated in Figure \ref{fig:step line}.

For the type II multiple orthogonal polynomials on the step line, we obtain a polynomial sequence with exactly one polynomial of degree $n$ for each $n\in\N$.
These are often referred to as \textit{$d$-orthogonal polynomials} (where $d$ is the number of measures, so $d=r$), as introduced in \cite{MaroniOrthogonalite}. 
In the case of $r=2$ measures, the  type II multiple orthogonality conditions \eqref{orthogonality conditions type II} on the step line correspond to say that if we set 
\begin{equation}\label{Pn stepline}
	P_{2n}(x) = P_{n,n}(x) \quad \text{and}\quad 
	P_{2n+1}(x) = P_{n+1,n}(x), 
\end{equation}
then  the polynomial sequence $\dis\polyseq$ 
is \textit{$2$-orthogonal} with respect to a pair of measures $(\mu_0,\mu_1)$: 
\begin{align}
\label{2-orthogonality conditions}
\int x^kP_n(x)\dd\mu_0(x)
=\begin{cases}
0, &\text{ if } n\geq 2k+1,\\
N_n\neq 0, &\text{ if } n=2k,
\end{cases}
\text{ and }
\int x^kP_n(x)\dd\mu_1(x)
=\begin{cases}
0, &\text{ if } n\geq 2k+2,\\
N_n\neq 0, &\text{ if } n=2k+1.
\end{cases}
\end{align}

In \eqref{Pn stepline} and throughout we have considered the step line to be the lower step line as illustrated in Figure \ref{fig:step line}. 
If we were to consider the polynomials on the upper step line, then these happened to be $2$-orthogonal with respect to the vector of measures 
$(\mu_1,\mu_0)$. 


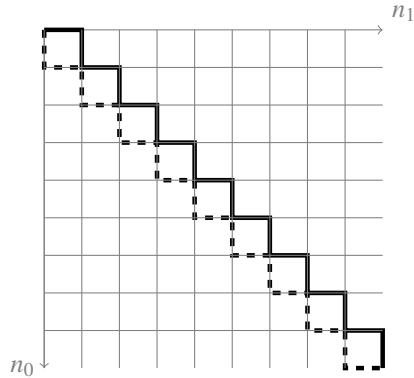
\begin{figure}[ht]
\centering
\begin{tikzpicture}[domain=0:2]
\draw[->,color=gray] (0,4.5) -- (4.5,4.5)  node[above right] {$n_1$};
\draw[->,color=gray] (0,4.5)  --  (0,0)  node[left] {$n_0$};
\draw[line width=0.6mm, color=black,solid] (0,4.5) -- (0.5,4.5) -- (0.5,4) -- (1, 4) -- (1,3.5) -- (1.5,3.5) -- (1.5, 3) -- (2,3) -- (2,2.5) -- (2.5,2.5) -- (2.5 , 2) -- (3,2) -- (3 , 1.5) -- (3 .5, 1.5) --(3.5 , 1) --(4 , 1) -- (4 , 0.5) --(4.5 , 0.5) -- (4.5, 0);
\draw[line width=0.6mm, color=black,dashed] (0,4.5) -- (0,4) -- (0.5,4) -- (0.5, 3.5) -- (1,3.5) -- (1,3) -- (1.5, 3) -- (1.5,2.5) -- (2,2.5) -- (2,2) -- (2.5 , 2) -- (2.5,1.5) -- (3,1.5) -- (3,1) -- (3.5, 1) -- (3.5 , 0.5) -- (4, 0.5) -- (4 , 0)-- (4.5, 0);
\draw[-, color=gray] (4,0) -- (4,4.5);
\draw[-, color=gray] (3.5,0) -- (3.5,4.5);
\draw[-, color=gray] (3,0) -- (3,4.5);
\draw[-, color=gray] (2.5,0) -- (2.5,4.5);
\draw[-, color=gray] (2,0) -- (2,4.5);
\draw[-, color=gray] (1.5,0) -- (1.5,4.5);
\draw[-, color=gray] (1,0) -- (1,4.5);
\draw[-, color=gray] (0.5,0) -- (0.5,4.5);
\draw[-, color=gray] (0,0) -- (0,4.5);

\draw[-, color=gray] (0,4) -- (4.5,4);
\draw[-, color=gray] (0,3.5) -- (4.5,3.5);
\draw[-, color=gray] (0,3) -- (4.5,3);
\draw[-, color=gray] (0,2.5) -- (4.5,2.5);
\draw[-, color=gray] (0,2) -- (4.5,2);
\draw[-, color=gray] (0,1.5) -- (4.5,1.5);
\draw[-, color=gray] (0,1) -- (4.5,1);
\draw[-, color=gray] (0,0.5) -- (4.5,0.5);
\end{tikzpicture}
\caption{Upper and lower step line for the multi-index  $(n_0,n_1)$ in solid and dashed black line, respectively, when $r=2$.
}
\label{fig:step line}
\end{figure}

There is a well-known connection between orthogonal polynomials and recurrence relations.
The spectral theorem for orthogonal polynomials (also known as Shohat-Favard theorem) states that a polynomial sequence $\{p_n\}_{n\in\N}$ is orthogonal with respect to some measure $\mu$ if and only if it satisfies a second order recurrence relation of the form
\begin{align*}
p_{n+1}(x)=(x-\beta_n)p_n(x)-\gamma_{n} p_{n-1}(x),
\end{align*}
with $\gamma_n\neq 0$, for all $n\geq 1$, and initial conditions $p_{-1}=0$ and $p_0=1$. 
Moreover, if $\beta_n\in\R$ and $\gamma_{n+1}>0$, for all $n\in\N$, then $\mu$ is a positive measure on the real line.

Multiple orthogonal polynomials also satisfy recurrence relations. In particular, when the multi-indexes lie on the step line, a polynomial sequence $\polyseq$ is $r$-orthogonal, and satisfies \eqref{2-orthogonality conditions}, if and only if it satisfies a recurrence relation of order $r+1$ of the form
\begin{align}
\label{recurrence relation for r-orthogonality}
P_{n+1}(x)=(x-\beta_n)P_n(x)-\sum_{j=1}^{r}\gamma_{n+1-j}^{\parameter{j}} P_{n-j}(x),
\end{align}
with $\gamma_{n}^{(r)}\neq 0$, for all $n\geq 1$, and initial conditions $P_{-(r-1)}=\cdots=P_{-1}=0$ and $P_0=1$, see \cite{MaroniOrthogonalite}.
Naturally, when $r=2$, the relation \eqref{recurrence relation for r-orthogonality} reduces to the third order recurrence relation
\begin{align}
\label{recurrence relation for a 2-OPS}
P_{n+1}(x)=(x-\beta_n)P_n(x)-\alpha_{n} P_{n-1}(x)-\gamma_{n-1} P_{n-2}(x),
\end{align}
with $\gamma_n\neq 0$, for all $n\geq 1$, and initial conditions $P_{-2}=P_{-1}=0$ and $P_0=1$.

For the type I multiple orthogonal polynomials on the step line for $r=2$ measures, 
we have $\dis\deg(A_n)\leq\floor{\frac{n-1}{2}}$ and $\dis\deg(B_n)\leq\floor{\frac{n}{2}}-1$, that is, $\deg(A_n)=m-1$, if $n=2m$ or $n=2m-1$, and $\deg(B_n)=m-1$, if $n=2m$ or $n=2m+1$. 
Moreover, assuming that there exists a positive measure $\mu$ and a pair of weight functions $\dis(w_0,w_1)$ such that $\dis\dd\mu_0(x)=w_0(x)\dd\mu(x)$ and $\dis\dd\mu_0(x)=w_1(x)\dd\mu(x)$, the type I function is
\begin{align}
Q_n(x)=A_n(x)w_0(x)+B_n(x)w_1(x)
\end{align}
and the orthogonality and normalisation conditions correspond to
\begin{align}
\int x^kQ_n(x)\dd\mu(x)
=\begin{cases}
0, &\text{ if } 0\leq k\leq n-2,\\
1, &\text{ if } k=n-1.
\end{cases}
\end{align}
For further information about multiple orthogonal polynomials and Nikishin systems, we refer to \cite[Ch.~23]{IsmailBook} and \cite{GuillermoSurvey}.

In Section \ref{Multiple orthogonality}, we prove that the weight functions $\W(x;a,b,c)$ and $\W(x;a,b,c+1)$ defined in \eqref{modified Tricomi weight definition} form a Nikishin system. This readily imply that the multiple orthogonal polynomials of both type I and type II with respect to these weight functions exist and are unique for every multi-index $\n=(n_0,n_1)\in\N^2$ and their zeros satisfy the properties mentioned before for Nikishin systems (and AT-systems in general). Then we obtain differential equations satisfied by these weight functions, which we use to deduce differential properties for the multiple orthogonal polynomials of both type I and type II on the step line (see Theorem \ref{differentiation formulas for the multiple orthogonal polynomials - theorem}) and a Rodrigues-type formula for the type I polynomials (see Theorem \ref{Rodrigues formula for the type I polynomials - theorem}).

Section \ref{Type II multiple orthogonal polynomials} is devoted to the characterisation of the type II multiple orthogonal polynomials on the step-line (ie, the {\it $2$-orthogonal} polynomial sequences). A remarkable property of these polynomials (a straightforward consequence of Theorem \ref{differentiation formulas for the multiple orthogonal polynomials - theorem}) is that they satisfy the so called {\it Hahn's property}, meaning that the sequence of its derivatives is again 2-orthogonal. As such, they stand as an example of a Hahn-classical $2$-orthogonal family. Our detailed characterisation of these polynomials includes: an explicit expression for these polynomials as a terminating generalised hypergeometric series, more precisely a $\dis{}_{2}F_{2}$ (see Theorem \ref{explicit formulas for the 2-orthogonal polynomials}); explicit  third order differential equation (in Theorem \ref{thm: diff eq}) as well as a third order recurrence relation (in Theorem \ref{recurrence relation satisfied by the 2-OPS and asymptotic behaviour of the recurrence coefficients}) to which these type II polynomials on the step-line  are a solution;  an asymptotic upper bound for their largest zeros; limiting relations between these polynomials and multiple orthogonal polynomials with res\-pect to weight functions involving the modified Bessel function of second kind $\BesselK{\nu}$ (see \eqref{Bessel weights definition}) studied in \cite{BenCheikhandDouak} and \cite{SemyonandWalter}. It turns out that each of the sequences of recurrence coefficients is unbounded and asymptotically periodic of period $2$. As such, we believe this is the first explicit example of a Nikishin system associated with such periodic unbounded recurrence coefficients.  

Earlier we mentioned generalised hypergeometric series, which are formally defined by
\begin{align}
\label{generalised hypergeometric function}
\Hypergeometric[z]{p}{q}{\alpha_1,\cdots,\alpha_p}{\beta_1,\cdots,\beta_q}
=\sum_{n=0}^{\infty}
\frac{\pochhammer{\alpha_1}\cdots\pochhammer{\alpha_p}}
{\pochhammer{\beta_1}\cdots\pochhammer{\beta_q}}
\frac{z^n}{n!} \ ,
\end{align}
where $p,q\in\N$, $z,\alpha_1,\cdots,\alpha_p\in\C$ and $\beta_1,\cdots,\beta_p\in\C\backslash\{-n : \ n\in\N\}$.
If one the parameters $\alpha_1,\cdots,\alpha_p$ is a non-positive integer the series \eqref{generalised hypergeometric function} terminates and defines a polynomial.


In Section \ref{Connection with Hahn-classical 3-fold symmetric 2-orthogonal polynomials} we explain that particular cases of the type II polynomials on the step line, characterised here, have appeared in \cite{AnaandWalter} as the components of $3$-fold symmetric Hahn-classical $2$-orthogonal polynomials on star-like sets. 
%
A polynomial sequence $\dis\polyseq$ is said to be \textit{$3$-fold symmetric} if, for any $n\in\N$,
\begin{align*}
P_n\left(\e^{\frac{2\pi i}{3}} x\right)=\e^{\frac{2n\pi i}{3}} P_n(x)
\text{ and }
P_n\left(\e^{\frac{4\pi i}{3}} x\right)=\e^{\frac{4n\pi i}{3}} P_n(x).
\end{align*}
This definition is equivalent to say that there exist three polynomial sequences $\dis\polyseq[P_n^\kk(x)]$, each supra indexed  with $k\in\{0,1,2\}$, which are called the cubic components of $\dis\polyseq$, such that 
\begin{equation}\label{cubic comp}
P_{3n+k}(x)=x^k P^\kk_n(x^3), \quad \text{for all }\quad n\in\N.
\end{equation}
In this section we give a result of independent interest, Theorem \ref{The cubic decomposition preserves the Hahn-classical property - theorem}, where we show that the cubic components of  $3$-fold symmetric Hahn-classical $2$-orthogonal polynomials are themselves Hahn-classical. 

The main contribution of this paper are the results in Sections \ref{Multiple orthogonality} and \ref{Type II multiple orthogonal polynomials} characterising multiple orthogonal polynomials with respect to the Nikishin system. The centre of the analysis is for the indices on the upper and lower step line (see Fig \ref{fig:step line}).   
The study of the multiple orthogonal polynomials with respect to the same system for indices out of the step line and, in particular, the study of the (standard) orthogonal polynomials with respect to the weight function $\W(x;a,b,c)$, defined by \eqref{modified Tricomi weight definition}, remains an open (and challenging) problem. Partly this is due to the fact that when the weight function is a solution to a second order differential equation, then known techniques to obtain closed or explicit formulas for recurrence coefficients of standard orthogonal polynomials is, up to now, an onerous task. An example of such weights are those studied here and given in \eqref{modified Tricomi weight definition} or those expressed in terms of Bessel functions in \eqref{Bessel weights definition}. Notwithstanding, a deep grasp of the multiple orthogonal polynomials on the step line is at the core of applications. The present investigation focus essentially on the latter. 
	
\section{Multiple orthogonality}\label{Multiple orthogonality}

The starting point of this investigation is on the weight function $\W(x;a,b,c)$  in \eqref{modified Tricomi weight definition}. The goal is to describe a system of multiple orthogonal polynomials with respect to $\W(x;a,b,c)$ and $\W(x;a,b,c+1)$. The first question to address is on whether such a system exists and, if so, whether it is unique. We are able to answer afirmatively to both issues because we are dealing with a Nikishin system of measures, as explained in Section \ref{A Nikishin system and uniqueness of the multiple orthogonal polynomials}. From this we want to move on to the characterisation of such system of polynomials. We succeed in doing so for the case where the indices lie on the step-line. Using key differential properties for the vector of weights, derived in Section \ref{Differential properties of the weight functions}, we obtain differential properties for the corresponding polynomials of both types in Section \ref{Differential properties of the multiple orthogonal polynomials}. We continue the analysis by providing a Rodrigues-type formula for the type I functions in Section \ref{Rodrigues-type formula for type I multiple orthogonal polynomials}. 
Concerning the type II, we defer their investigation to Section \ref{Type II multiple orthogonal polynomials}.

\subsection{Nikishin system}
\label{A Nikishin system and uniqueness of the multiple orthogonal polynomials}
The vector of weight functions $\dis\left( {\W(x;a,b,c)}\ , \ {\W(x;a,b,c+1)}\right)$ forms a Nikishin system, as stated in Theorem \ref{Nikishin system}. An important consequence of this result is that both type I and II multiple orthogonal polynomials with respect to the weight functions appearing on Theorem \ref{Nikishin system} exist and are unique for any multi-index $\dis(n_0,n_1)\in\N^2$.
Moreover, the type I multiple orthogonal polynomials $A_{(n_0,n_1)}$ and $B_{(n_0,n_1)}$ have degree exactly $n_0-1$ and $n_1-1$, respectively, and the type II multiple orthogonal polynomial $P_{(n_0,n_1)}$ has $n_0+n_1$ positive real simple zeros that satisfy an interlacing property: there is always a zero of $P_{(n_0,n_1)}$ between two consecutive zeros of $P_{(n_0+1,n_1)}$ or $P_{(n_0,n_1+1)}$.

On the one hand, the Nikishin property can be deduced through the connection between continued fractions and Stieltjes transforms, by guaranteeing the existence of a generating measure, as described in \eqref{Nikishin system - ratio of the measures} . On the other hand, this property can be proved by providing an integral representation for the generating measure $\sigma$ in \eqref{Nikishin system - ratio of the measures} for this Nikishin system, and this is explained at the end of this subsection.  

To start with, we recall some properties from continued fractions and we follow the notation in  \cite{ContinuedFractionsForSpecialFunctions} to describe a continued fraction: 
\begin{align}
\label{continued fraction}
\beta_0+\contfrac{\alpha_j}{\beta_j}
:=\beta_0+\cfrac{\alpha_1}{\beta_1+\cfrac{\alpha_2}{\beta_2+\cdots}}.
\end{align}
Particularly relevant are the so-called {\it S-fractions} or {\it Stieltjes continued fractions}, which are obtained if in \eqref{continued fraction} we set  $\beta_0=0$, $\alpha_n=1$, $\beta_{2n-1}=a_{2n-1}\, z$ and $\beta_{2n}=a_{2n}$, with $a_{2n-1},a_{2n}\in\R^+$, for $n\geq 1$, to get: 
\begin{align}
\label{S-fraction}
\cfrac{1}{a_1z+\cfrac{1}{a_2+\cdots+\cfrac{1}{a_{2n-1}z+\cfrac{1}{a_{2n}+\cdots}}}} \ .
\end{align}

Stieltjes showed in \cite{Stieltjesmemoir} that S-fractions can be represented as a Stieltjes transform of a measure with support in $(-\infty,0]$, that is, an integral of the form
\begin{align}
\label{Stieltjes integral}
\int_{-\infty}^{0}\frac{\dd\sigma(-t)}{x-t}=\int_{0}^{\infty}\frac{\dd\sigma(u)}{x+u},
\end{align}
where $\sigma$ is a non decreasing bounded function such that $\dis\sigma(0)=0$ and $\dis\lim_{u\to\infty}\sigma(u)=\frac{1}{a_1}$.

Another special type of continued fraction, known as J-fraction, is obtained if we set, in \eqref{continued fraction}, $\beta_0=0$ and, for each $n\geq 1$, $\alpha_n=c_n^2$ and $\beta_n=z+b_n$, for some $c_n,b_n\in\C$:
\begin{align}
\label{J-fraction}
\cfrac{c_1^2}{z+b_1+\cfrac{c_2^2}{z+b_2+\cdots+\cfrac{c_n^2}{z+b_{n}+\cdots}}}.
\end{align}

If every $c_n,b_n\in\R^+$ then the J-fraction generated by them can be obtained by contraction from a S-fraction (see \cite{Stieltjesmemoir}) and, as a result, it can also be represented as a Stieltjes transform with support in $(-\infty,0]$.
These results about continued fractions and Stieltjes transforms can also be found in \cite[Ch.~13]{WallContinuedFractions} and they are used here to prove the following result.

\begin{theorem}
\label{Nikishin system}
Let $a,b,c\in\R$  be such that $a>-1$, $b>-1$ and $c>\max\{0,a-b\}$. 
Then the vector of weight functions $\dis\left({\W(x;a,b,c)} \ , \ {\W(x;a,b,c+1)}\right)$ defined by \eqref{modified Tricomi weight definition} forms a Nikishin system.
\end{theorem}

\begin{proof}
Following the definition of $\dis\W(x;a,b,c)$, observe that 
\begin{equation}\label{ratio of Ws}
\frac{\W(x;a,b,c+1)}{\W(x;a,b,c)}=\frac{(c+b+1)\KummerU{c+1}{a-b+1}}{\KummerU{c}{a-b+1}}.
\end{equation}
According to \cite[Eq.~13.3.7]{DLMF}, we have 
\[
c(c+b-a)\KummerU{c+1}{a-b+1}=(x+2c+b-a-1)\KummerU{c}{a-b+1}-\KummerU{c-1}{a-b+1}, 
\]
and, based on  \cite[Eq.~16.1.20]{ContinuedFractionsForSpecialFunctions}, we also have 
\[
\frac{\KummerU{c-1}{a-b+1}}{\KummerU{c}{a-b+1}}
=x+2c+b-a-1-\contfrac{(c+j-1)(c+b-a+j-1)}{x+2c+b-a+2j-1}.
\]
Combining the three latter equations, we obtain
\begin{equation}
\frac{\W(x;a,b,c+1)}{\W(x;a,b,c)}
=\frac{c+b+1}{c(c+b-a)}\contfrac{(c+j-1)(c+b-a+j-1)}{x+2c+b-a+2j-1}.
\end{equation}

As $c>\max\{0,a-b\}$ then, for any $j\geq 1$, we have $c+j-1,c+b-a+j-1,2c+b-a+2j-1>0$. 
Hence we know that there exists a measure $\sigma$ such that
\begin{equation}
\label{continued fraction as a Stieltjes transform}
\contfrac{(c+j-1)(c+b-a+j-1)}{x+2c+b-a+2j-1}
=\int_{-\infty}^{0}\frac{\dd\sigma(t)}{x-t}
\end{equation}
and, as a result, we deduce 
\begin{equation}
\label{ratio of the weights as a Stieltjes transform}
\frac{\W(x;a,b,c+1)}{\W(x;a,b,c)}
=\int_{-\infty}^{0}\frac{c+b+1}{c(c+b-a)}\frac{\dd\sigma(t)}{x-t}, 
\end{equation}
which shows that the vector of measures  $\dis\left({\W(x;a,b,c)} \ , \ {\W(x;a,b,c+1)}\right)$  forms a Nikishin system.\end{proof}
\vspace*{0,1 cm}

Now we find the generating measure $\sigma$ in \eqref{continued fraction as a Stieltjes transform} and \eqref{ratio of the weights as a Stieltjes transform}. 
In \cite{IsmailandKelker} it is proved that the integral representation
\begin{equation}
\label{integral representation for the Tricomi function ratio}
\frac{\KummerU{\alpha+1}{\beta+1}}{\KummerU{\alpha}{\beta}}
=\int_{-\infty}^{0}\frac{(-t)^{-\beta}e^{-t}\left|\KummerU[t]{\alpha}{\beta}\right|^{-2}\dd{\mathrm{t}}} {(x-t)\Gamma(\alpha+1)\Gamma(\alpha-\beta+1)}
\end{equation}
is valid for $\alpha>0$ and $\beta<1$.
Moreover, the condition $\beta<1$ can be replaced by $\alpha-\beta+1>0$ 
(which is a weaker condition because $\alpha>0$) as this change doesn't interfere with the proof because the conditions $\alpha>0$ and $\alpha-\beta+1>0$ are sufficient to guarantee that $\KummerU{\alpha}{\beta}$ has no zeros in the region $|\mathrm{arg}x|<\pi$ (see \cite[\S 13.9(i)]{DLMF}).
We take $\alpha=c$ and $\beta=a-b+1$ in \eqref{integral representation for the Tricomi function ratio}, then we recall the definition of $\dis\W(x;a,b,c)$  in \eqref{modified Tricomi weight definition} to obtain the following integral representation 
\begin{equation}
\frac{\W(x;a,b,c+1)}{\W(x;a,b,c)}
=(c+b+1)\frac{\KummerU{c+1}{a-b+1}}{\KummerU{c}{a-b+1}}
=\int_{-\infty}^{0}\frac{(c+b+1)(-t)^{b-a-1}e^{t}\left|\KummerU[t]{c}{a-b+1}\right|^{-2}\dd{\mathrm{t}}} {(x-t)\Gamma(c+1)\Gamma(c-b+a)},
\end{equation}
valid for $c>\max\{0,b-a\}$. 


\subsection{Differential properties of the weight functions}
\label{Differential properties of the weight functions}

From this point forth,  we will index the vector of weights $\dis\left( {\W(x;a,b,c)}\ , \ {\W(x;a,b,c+1)}\right)$, as defined in \eqref{modified Tricomi weight definition}, with an extra parameter $d\in\{0,1\}$, by considering
\begin{equation}\label{Wvec}
	\Wvec^\parameter{d}(x;a,b,c)
	:=\dis\left[ \begin{array}{c} \W(x;a,b,c+d) \\  \W(x;a,b,c+1-d) \end{array}\right],\quad \text{with}\quad d\in\{0,1\}. 
\end{equation} 
The parameter $d\in\{0,1\}$ embodies  the flip between the lower and the upper step line indexes of the corresponding multiple orthogonal polynomials of both types. As a consequence, if $\ \dis\polyseq[P_n^\parameter{d}(x;a,b,c)]$ is the monic $2$-orthogonal polynomial sequence and $\dis Q_n^\parameter{d}(x;a,b,c)$ the type I function for the index of length $n$ on the step-line for $\dis\Wvec^\parameter{d}(x;a,b,c)$, then 
\[ P_{2n}^\parameter{1}(x;a,b,c)= P_{2n}^\parameter{0}(x;a,b,c)\quad
\text{and}\quad 
Q_{2n}^\parameter{1}(x;a,b,c)= Q_{2n}^\parameter{0}(x;a,b,c).\]
There are further motivations for the introduction of this parameter $d$. Under the action of the derivative operator, the multiple orthogonal system for $\dis\Wvec^\parameter{d}(x;a,b,c)$  bounces from the lower to the upper step line (and reciprocally) with shifted parameters, as perceivable in Theorem \ref{differentiation formulas for the multiple orthogonal polynomials - theorem}. 
A result that comes as a consequence of Theorem \ref{canonical matrix differential equation satisfied by the weights - theorem} where the vector of weights \eqref{Wvec} is described as a solution to a matrix first order differential equation. Its structure fits into the category of  Hahn-classical type vector of weights, in the sense expounded in \cite{DouakandMaroniClassiquesDeDimensionDeux}. 
Beforehand, in Proposition \ref{2ndOrderDiffEqForTheWeight} we describe the weight function $\dis\W(x;a,b,c)$ in \eqref{modified Tricomi weight definition}  as a solution to a second-order differential equation. \\

\begin{proposition}
\label{2ndOrderDiffEqForTheWeight}
Let $a,b,c\in\R$ such that $a>-1$, $b>-1$ and $c>\max\{0,a-b\}$. Then $\dis\W(x;a,b,c)$ defined in \eqref{modified Tricomi weight definition} satisfies the differential equation 
\begin{align}
\label{differential equation of second order satisfied by W_0}
x^2\W''(x;a,b,c)
+\left(x-(a+b-1)\right)x\W'(x;a,b,c)
+\left(ab-(c+b-1)x\right)\W(x;a,b,c)
=0. 
\end{align}
\end{proposition}

\begin{proof}
Let  $\dis\lambda=\frac{\Gamma(c+b+1)}{\Gamma(a+1)\Gamma(b+1)}$ and $\dis\U(x)=\KummerU{c}{a+1-b}$.
We differentiate \eqref{modified Tricomi weight definition} with respect to $x$ to obtain 
\begin{align}\label{W prime}
\DiffOp\W(x;a,b,c)=\lambda\e^{-x}x^{a-1}\left(x \ \dis\U'(x) +(a-x) \dis\U(x)\right). 
\end{align} 
Another differentiation brings 
\[
\frac{\mathrm{d}^2}{\mathrm{d}x^2}\W(x;a,b,c)=\lambda\e^{-x}x^{a-2}\left(x^2\U''(x)+2(a-x)x\U'(x)+\left(x^2-2ax+a(a-1)\right)\U(x)\right).
\]
Based on \eqref{Kummer equation}, we have $\dis x\U''(x)=(x-a-1+b)\U'(x)+c\U(x)$, so that we have   
\[
\frac{\mathrm{d}^2}{\mathrm{d}x^2}\W(x;a,b,c)=\lambda\e^{-x}x^{a-2}\left((a+b-1-x)x\ \U'(x)+\left(x^2+(c-2a)x+a(a-1)\right)\U(x)\right).
\]
Finally, combining the latter expression with the definition of $\W(x;a,b,c)$ and \eqref{W prime}, we deduce \eqref{differential equation of second order satisfied by W_0}.
\end{proof}

To prove Theorem \ref{canonical matrix differential equation satisfied by the weights - theorem} we need the following  technical result, which allows us to write $\dis\DiffOp\left(x\ \W(x;a,b,c+d)\right)$, with $d\in\{0,1\}$, as a linear combination of $\dis\W(x;a,b,c)$ and $\dis\W(x;a,b,c+1)$.

\begin{lemma}
\label{initial matrix differential relations between the weights}
Let $a,b,c\in\R$ such that $a>-1$, $b>-1$ and $c>\max\{0,a-b\}$. If $\Wvec^\parameter{d}(x;a,b,c)$ is the vector defined in \eqref{Wvec}, then
\begin{equation}
\label{matrix differential relation between the weights d=0}
\DiffOp\left(x\Wvec^\parameter{0}(x;a,b,c)\right)
=-\matrixtwobytwo{x+c-a-1}{-\frac{c(c+b-a)}{c+b+1}}{c+b+1}{-(c+b+1)}\Wvec^\parameter{0}(x;a,b,c)
\end{equation}
and
\begin{equation}
\label{matrix differential relation between the weights d=1}
\DiffOp\left(x\Wvec^\parameter{1}(x;a,b,c)\right)
=-\matrixtwobytwo{-(c+b+1)}{c+b+1}{-\frac{c(c+b-a)}{c+b+1}}{x+c-a-1}\Wvec^\parameter{1}(x;a,b,c).
\end{equation}
\end{lemma}

\begin{proof}
Using the fact that $\dis\DiffOp\KummerU{c}{a-b+1}=-c\,\KummerU{c+1}{a-b+2}$ (see \cite[Eq.~13.3.22]{DLMF}), in \eqref{W prime} gives 
\[ 
\DiffOp\left(\W(x;a,b,c)\right)=
\lambda\e^{-x}x^{a-1}\left(-cx\,\KummerU{c+1}{a-b+2}+(a-x)\,\KummerU{c}{a-b+1}\right).
\]
Moreover, since $\dis x\KummerU{c+1}{a-b+2}=\KummerU{c}{a-b+1}-(c+b-a)\KummerU{c+1}{a-b+1}$ (see \cite[Eq.~13.3.10]{DLMF}), it follows that 
\[
\DiffOp\left(\W(x;a,b,c)\right)=
\lambda\e^{-x}x^{a-1}\left(c(c+b-a)\KummerU{c}{a-b+2}+(-c+a-x)\,\KummerU{c}{a-b+1}\right).
\]
As a result of the definition of $\W(x;a,b,c)$ in \eqref{modified Tricomi weight definition}, we obtain 
\begin{equation}
\label{W0'(c) as combination of W0(c) and W0(c+1)}
\DiffOp\left(x\W(x;a,b,c)\right)=-(x+c-a-1)\W(x;a,b,c)+\frac{c(c+b-a)}{c+b+1}\W(x;a,b,c+1), 
\end{equation}
which, after taking the parameter shift $c\to c+1$, reads as
\[
\DiffOp\left(x\W(x;a,b,c+1)\right)=-(x+c-a)\W(x;a,b,c+1)+\frac{(c+1)(c+b-a+1)}{c+b+2}\W(x;a,b,c+2).
\]
The last term in the latter expression can actually be written as 
\[
\frac{(c+1)(c+b-a+1)}{c+b+2}\W(x;a,b,c+2)=(x+2c+b-a+1)\W(x;a,b,c+1)-(c+b+1)\W(x;a,b,c).
\]
because of the identity (see  \cite[Eq.~13.3.7]{DLMF})
\[\dis(c+1)(c+b-a+1)\KummerU{c+2}{a-b+1}=(x+2c+b-a+1)\KummerU{c+1}{a-b+1}-\KummerU{c}{a-b+1}.\] 
As a result, we have 
\begin{equation}
\label{W0'(c+1) as combination of W0(c) and W0(c+1)}
\DiffOp\left(x\W(x;a,b,c+1)\right)=(c+b+1)\W(x;a,b,c+1)-(c+b+1)\W(x;a,b,c).
\end{equation}
Finally, the relations \eqref{matrix differential relation between the weights d=0} and \eqref{matrix differential relation between the weights d=1} are a direct consequence of \eqref{W0'(c) as combination of W0(c) and W0(c+1)} and \eqref{W0'(c+1) as combination of W0(c) and W0(c+1)}.
\end{proof}

Based on the previous, we can write the vector of weights $\Wvec^\parameter{d}(x;a,b,c)$ as a solution to a matrix first order equation of Pearson type. More precisely, we have:

\begin{theorem}
\label{canonical matrix differential equation satisfied by the weights - theorem}
Let $a,b,c\in\R$ such that $a>-1$, $b>-1$ and $c>\max\{0,a-b\}$ and $\Wvec^\parameter{d}(x;a,b,c)$ as in \eqref{Wvec}. Then
\begin{equation}
\label{canonical matrix differential equation satisfied by the weights d=0}
\DiffOp\left(x\,\Phi^\parameter{d}(x;a,b,c)\Wvec^\parameter{d}(x;a,b,c)\right)
+\Psi^\parameter{d}(x;a,b,c)\Wvec^\parameter{d}(x;a,b,c)=0,
\end{equation}
where
\begin{align*}
& \Phi^\parameter{0}
=\matrixtwobytwo{0}{\frac{c+b+2}{(a+1)(b+1)}}{\frac{c+b+1}{(a+1)(b+1)}}{0}
\ , \quad
\Phi^\parameter{1}(x)
=\matrixtwobytwo{\frac{c+b+2}{(a+1)(b+1)}}{0}{\eta(x+2c+b-a+1)}{-(c+b+1)\eta}, \\[0.3cm]
&
	\Psi^\parameter{0}(x)
=\matrixtwobytwo{\frac{(c+b+1)(c+b+2)}{(a+1)(b+1)}}{-\frac{(c+b+1)(c+b+2)}{(a+1)(b+1)}}
{\frac{c+b+1}{(a+1)(b+1)}\left(x+c-a-1\right)}{-\frac{c(c+b-a)}{(a+1)(b+1)}} , \\[0.3cm]
&	\Psi^\parameter{1}(x)
=\matrixtwobytwo{-\frac{(c+b+1)(c+b+2)}{(a+1)(b+1)}}{\frac{(c+b+1)(c+b+2)}{(a+1)(b+1)}}
{-(c+b+2)\eta\left(x+\xi\right)}{(c+b+1)(c+b+2)\eta} , 
\end{align*}
%
with $\dis\eta=\frac{(c+b+2)(c+b+3)}{(a+1)(b+1)(c+1)(c+b-a+1)}$
 and $\dis\xi=\frac{b^2-ab+2bc+2b+c^2+3c-a+1}{c+b+2}$.
\vspace*{0,1 cm}

Moreover,
\begin{equation}
\label{shift on the parameters}
x\,\Phi^\parameter{d}(x;a,b,c)\Wvec^\parameter{d}(x;a,b,c)=\Wvec^\parameter{1-d}(x;a+1,b+1,c+d).
\end{equation} 

\end{theorem}

\begin{proof}
Let $\dis\Wvec^\parameter{d}(x)=\Wvec^\parameter{d}(x;a,b,c)$, $d\in\{0,1\}$
and 
\[\dis\Omega^\parameter{0}(x)=\matrixtwobytwo{x+c-a-1}{-\frac{c(c+b-a)}{c+b+1}}{c+b+1}{-(c+b+1)}
\quad\text{
and}
\quad \dis\Omega^\parameter{1}(x)=\matrixtwobytwo{-(c+b+1)}{c+b+1}{-\frac{c(c+b-a)}{c+b+1}}{x+c-a-1}.\]  
The equations \eqref{matrix differential relation between the weights d=0} and \eqref{matrix differential relation between the weights d=1} in  Lemma \ref{initial matrix differential relations between the weights} read as follows
\[
	\DiffOp\left(x\Wvec^\parameter{d}(x;a,b,c)\right)
=- \Omega^\parameter{d}(x)\Wvec^\parameter{d}(x;a,b,c) .
\]
We multiply the latter by $\Phi^\parameter{d}(x)$ and we obtain 
\[ 
	\DiffOp\left(x\ \Phi^\parameter{d}(x)\ \Wvec^\parameter{d}(x;a,b,c)\right)
=\left(x\,\DiffOp\left(\Phi^\parameter{d}(x)\right)-\Phi^\parameter{d}(x)\Omega^\parameter{d}(x)\right)\Wvec^\parameter{d}(x) , 
\]
which corresponds to \eqref{canonical matrix differential equation satisfied by the weights d=0}, after observing that 
\[\dis\Psi^\parameter{0}(x)=\Phi^\parameter{0}\Omega^\parameter{0}(x)
\quad \text{ and }\quad 
\dis\Psi^\parameter{1}(x)=\Phi^\parameter{1}(x)\Omega^\parameter{1}(x)-x\,\DiffOp\left(\Phi^\parameter{1}(x)\right), \]
or, equivalently, that 
\[
	\dis\Psi^\parameter{d}(x)=\Phi^\parameter{d}(x)\Omega^\parameter{d}(x)-x\,\DiffOp\left(\Phi^\parameter{d}(x)\right).
\]
%
%
%

Now, let 
\[ \dis\twovector{\V_0^\parameter{d}(x)}{\V_1^\parameter{d}(x)}
=x\,\Phi^\parameter{d}(x)\twovector{\W(x;a,b,c+d)}{\W(x;a,b,c+1-d)}
	\quad \text{with}\quad d\in\{0,1\}.
\]
In order to prove \eqref{shift on the parameters}, we need to check that $\dis\twovector{\V_0^\parameter{d}(x)}{\V_1^\parameter{d}(x)}=\twovector{\W(x;a+1,b+1,c+1)}{\W(x;a+1,b+1,c+2d)}$. Indeed, we have 
\[\dis\V_0^\parameter{1}(x)=\V_0^\parameter{0}(x)
=\frac{\Gamma(c+b+3)}{\Gamma(a+2)\Gamma(b+2)}\e^{-x}x^{a+1}\KummerU{c+1}{a-b+1}
=\W(x;a+1,b+1,c+1),\]
as well as 
\[\dis\V_1^\parameter{0}(x)
=\frac{\Gamma(c+b+2)}{\Gamma(a+2)\Gamma(b+2)}\e^{-x}x^{a+1}\KummerU{c}{a-b+1}
=\W(x;a+1,b+1,c).\]
Besides, 
\[\dis\V_1^\parameter{1}(x)
=\frac{\Gamma(c+b+4)}{\Gamma(a+2)\Gamma(b+2)}\,\e^{-x}x^{a+1}\,
\frac{(x+2c+b-a+1)\KummerU{c+1}{a-b+1}-\KummerU{c}{a-b+1}}{(c+1)(c+b-a+1)},\]
and, since $\dis(x+2c+b-a+1)\KummerU{c+1}{a-b+1}-\KummerU{c}{a-b+1}=(c+1)(c+b-a+1)\KummerU{c+2}{a-b+1}$ (see \cite[Eq.~13.3.7]{DLMF}), it can be written as
\[\dis\V_1^\parameter{1}(x)
=\frac{\Gamma(c+b+4)}{\Gamma(a+2)\Gamma(b+2)}\,\e^{-x}x^{a+1}\,\KummerU{c+2}{a-b+1}
=\W(x;a+1,b+1,c+2).\]
\end{proof}

\subsection{Differential properties of the multiple orthogonal polynomials}
\label{Differential properties of the multiple orthogonal polynomials}
The main result of this section is Theorem \ref{differentiation formulas for the multiple orthogonal polynomials - theorem}, where we present a differential relation for multiple orthogonal polynomials on the step line of type II   in \eqref{differentiation formula for type II polynomials} and of type I in \eqref{differentiation formula for type I polynomials}.
More precisely, we show that the differentiation with respect to the variable gives a shift on the parameters as well as on the index. 
Therefore, we can see these polynomials as part of the Hahn-classical family, since both type II and type I multiple orthogonal polynomials on the step line satisfy the Hahn-classical property. 

To derive this theorem, we first prove Propositions \ref{differentiation formula for type II polynomials - general result} and \ref{differentiation formula for type I polynomials - general result} that give us differential properties for type II and type I multiple orthogonal polynomials on the step line in more general contexts.
Proposition \ref{differentiation formula for type II polynomials - general result} is a consequence of the alternative characterisation of the Hahn-classical property for $2$-orthogonal polynomials (ie multiple orthogonal polynomials on the step line) derived by Douak and Maroni in \cite{DouakandMaroniClassiquesDeDimensionDeux} (see also \cite[Prop.~6.2]{MaroniSemiclassical}). Here, we present an alternative proof, restricting ourselves to the use of weight functions instead of linear functionals.  
Incidentally, evoking similar arguments,  Proposition \ref{differentiation formula for type I polynomials - general result} is an analogous result for type I polynomials, which we believe to be new. 

\begin{proposition}
\label{differentiation formula for type II polynomials - general result}
Let $\dis\overline{w}(x)=\twovector{w_0(x)}{w_1(x)}$ be a vector of weight functions satisfying a differential equation 
\begin{equation}
\label{canonical matrix differential equation satisfied by general weights type II}
\DiffOp\left(x\Phi(x)\overline{w}(x)\right)+\Psi(x)\overline{w}(x)=0,
\end{equation}
with
$\dis\Phi(x)=\matrixtwobytwo{\phi_{00}}{\phi_{01}}{\varphi(x)}{\phi_{11}}$
and \vspace*{0,1 cm}
$\dis\Psi(x)=\matrixtwobytwo{\eta_0}{\eta_1}{\psi(x)}{\xi}$,
for constants $\phi_{00}$, $\phi_{01}$, $\phi_{11}$, $\eta_0$, $\eta_1$ and $\xi$ and polynomials $\varphi$ and $\psi$ such that $\deg\varphi\leq 1$ and $\deg\psi=1$.
Suppose that all multi-indices on the step-line are normal with respect to both $\dis\overline{w}(x)$ and $\dis x\Phi(x)\overline{w}(x)$ and let $\dis\polyseq$ be the $2$-orthogonal polynomial sequence with respect to  $\dis\overline{w}(x)$.
Then $\polyseq[\frac{1}{n+1}\DiffOp\left(P_{n+1}(x)\right)]$ is $2$-orthogonal with respect to $\dis x\Phi(x)\overline{w}(x)$.
\end{proposition}

\begin{proof}
Let $\dis\overline{v}(x)=x\Phi(x)\overline{w}(x)=\twovector{v_0(x)}{v_1(x)}$.
Based on the assumption \eqref{canonical matrix differential equation satisfied by general weights type II}, we have  $\dis\DiffOp\left(\overline{v}(x)\right)=-\Psi(x)\overline{w}(x)$ so that 
\begin{align*}
\DiffOp\left(x^k\overline{v}(x)\right)=x^k\big(k\Phi(x)-\Psi(x)\big)\overline{w}(x)
\end{align*}
holds for any $k\in\N$, and it amounts to the same as 
\begin{align*}
\DiffOp\left(x^kv_0(x)\right)
=x^k\Big(\left(k\phi_{00}-\eta_0\right)w_0(x)+\left(k\phi_{01}-\eta_1\right)w_1(x)\Big)
\end{align*}
and
\begin{align*}
\DiffOp\left(x^kv_1(x)\right)
=x^k\Big(\left(k\varphi(x)-\psi(x)\right)w_0(x)+\left(k\phi_{11}-\xi\right)w_1(x)\Big).
\end{align*}
Performing integration by parts and then using the latter identities, we respectively have
\begin{align*}
\int_{0}^{\infty}x^kP'_{n+1}(x)v_0(x)\dx
=\int_{0}^{\infty}\left(\eta_0-k\phi_{00}\right)x^kP_{n+1}(x)w_0(x)\dx
+\int_{0}^{\infty}\left(\eta_1-k\phi_{01}\right)x^kP_{n+1}(x)w_1(x)\dx
\end{align*}
and
\begin{align*}
\int_{0}^{\infty}x^kP'_{n+1}(x)v_1(x)\dx
=\int_{0}^{\infty}\left(\psi(x)-k\varphi(x)\right)x^kP_{n+1}(x)w_0(x)\dx
+\int_{0}^{\infty}\left(\xi-k\phi_{11}\right)x^kP_{n+1}(x)w_1(x)\dx, 
\end{align*}
which are valid for any $k\in\N$. 

Arguing now with the  $2$-orthogonality of $\dis\polyseq[P_n(x)]$ with respect to $\dis\overline{w}(x)$, combined with the degrees of $\varphi$ and $\psi$ not being greater than $1$, we conclude that the polynomial sequence $\dis\polyseq[\frac{1}{n+1}P'_{n+1}(x)]$ is necessarily $2$-orthogonal with respect to $\overline{v}(x)$.
\end{proof}

A similar result can be deduced regarding multiple orthogonality of type I. 

\begin{proposition}
\label{differentiation formula for type I polynomials - general result}
Let $\dis\overline{w}(x)$ be a vector of weight functions satisfying \eqref{canonical matrix differential equation satisfied by general weights type II} with $\Phi(x)$ and $\Psi(x)$ being two polynomial matrices as described in Proposition \ref{differentiation formula for type II polynomials - general result}. 
Suppose that all multi-indices on the step-line are normal with respect to both $\dis\ \overline{w}(x)$ and $\dis\  x\,\Phi(x)\overline{w}(x)$ and let $ Q_n(x)$ be the type I function for the index of length $n$ on the step-line with respect to $\  x\,\Phi(x)\overline{w}(x)$.
\vspace*{0,1 cm}
Then $  \left( -\frac{1}{n}\DiffOp Q_n(x)\right)$ is the type I function for the index of length $n+1$ on the step line with respect to $\dis\overline{w}(x)$.
\end{proposition}

\begin{proof} 
By a simple integration and then by definition of $\dis Q_n(x)$ we have 
\begin{align*}
\int_{0}^{\infty}Q'_n(x)\dx=Q_n(x)\Big|_{0}^{\infty}=0,
\end{align*}
whilst, after performing integration by parts to then argue with the definition of  $\dis Q_n(x)$, we obtain 
\begin{align*}
\int_{0}^{\infty}x^{k+1}Q'_n(x)\dx
=-(k+1)\int_{0}^{\infty}x^kQ_n(x)\dx
=\begin{cases}
 0, &\text{ if } 0\leq k\leq n-2,\\
-n, &\text{ if } k=n-1.
\end{cases}
\end{align*}
Hence, it follows  
\begin{align*}
\int_{0}^{\infty}-x^j\,\frac{Q'_n(x)}{n}\dx
=\begin{cases}
0, & \text{ if }0\leq j\leq n-1,\\
1, & \text{ if }j=n.
\end{cases}
\end{align*}

Therefore it is sufficient to show that there are polynomials $A_{n+1}(x)$ and $B_{n+1}(x)$ such that 
\begin{equation}
\label{linear combination type I}
-\frac{1}{n}Q'_n(x)=A_{n+1}(x)w_0(x)+B_{n+1}(x)w_1(x)
\quad \text{for all}\quad n\in\N,
\end{equation}
and 
\begin{equation}
\label{conditions on the degrees of the type I polynomials}
\deg\left(A_{2m}(x)\right),\deg\left(B_{2m}(x)\right),\deg\left(B_{2m+1}(x)\right)\leq m-1
\text{ and }
\deg\left(A_{2m+1}(x)\right)\leq m,
\quad \text{for any}\quad m\in\N,
\end{equation}
because this implies that $\left(A_{n+1}(x),B_{n+1}(x)\right)$ is the vector of type I multiple orthogonal polynomials for the index of length $n+1$ on the step-line with respect to $\dis\overline{w}(x)$. Consequently, this means that $ \left( -\frac{1}{n}\DiffOp Q_n(x)\right)$ is the type I function for the index of length $n+1$ on the step line with respect to $\dis\overline{w}(x)$.

Consider the vector of weights $\dis\overline{v}(x)=x\Phi(x)\overline{w}(x)$ and let $\dis\overline{v}(x)=\left[\begin{array}{c} v_0(x) \\ v_1(x) \end{array}\right]$. 
Then, 
\begin{subequations}
\begin{align}
\label{relations between weights 1}
v_0(x)=x\left(\phi_{00}\,w_0(x)+\phi_{01}\,w_1(x)\right)
\text{ and }
v_1(x)=x\left(\varphi(x)w_0(x)+\phi_{11}\,w_1(x)\right).
\end{align}
By virtue of equation \eqref{canonical matrix differential equation satisfied by general weights type II}, we have $\dis\DiffOp\left(\overline{v}(x)\right)=-\Psi(x)\overline{w}(x)$, which means 
\begin{align}
\label{relations between weights 2}
v'_0(x)=-\eta_0 w_0(x)-\eta_1 w_1(x)
\text{ and }
v'_1(x)=-\psi(x) w_0(x)-\xi w_1(x).
\end{align}
\end{subequations}

For any $n\in\N$, let $\dis\left(C_n(x),D_n(x)\right)$ be the vector of type I multiple orthogonal polynomials for the index of length $n$ on the step-line with respect to $\dis\overline{v}(x)$.
Then, by definition of the type I function,
\begin{align*}
Q_n(x)=C_n(x)v_0(x)+D_n(x)v_1(x),
\end{align*}
with
\begin{equation*}
\deg\left(C_{2m}(x)\right),\deg\left(D_{2m}(x)\right),\deg\left(D_{2m+1}(x)\right)\leq m-1
\text{ and }
\deg\left(C_{2m+1}(x)\right)\leq m,
\text{ for any }m\in\N. 
\end{equation*}

Differentiating the expression for $Q_n(x)$, we obtain
\begin{align*}
Q'_n(x)=C'_n(x)v_0(x)+A_n(x)v'_0(x)+D'_n(x)v_1(x)+D_n(x)v'_1(x)
\end{align*}
hence, using \eqref{relations between weights 1} and \eqref{relations between weights 2}, we derive
\begin{align*}
Q'_n(x)&
=\Big(\phi_{00}x\,C'_n(x)-\eta_0 C_n(x)+x\varphi(x)D'_n(x)-\psi(x)D_n(x)\Big)w_0(x) \\&
+\Big(\phi_{01}x\,C'_n(x)-\eta_1 C_n(x)+\phi_{11}x\,D'_n(x)-\xi D_n(x)\Big)w_1(x).
\end{align*}
The uniqueness of type I multiple orthogonal polynomials leads to \eqref{linear combination type I} where 
\begin{align*}
A_{n+1}(x)=-\frac{1}{n}\Big(\phi_{00}x\,C'_n(x)-\eta_0 C_n(x)+x\varphi(x)D'_n(x)-\psi(x)D_n(x)\Big)
\end{align*}
and 
\begin{align*}
B_{n+1}(x)=-\frac{1}{n}\Big(\phi_{01}x\,C'_n(x)-\eta_1 C_n(x)+\phi_{11}x\,D'_n(x)-\xi D_n(x)\Big).
\end{align*}
Finally, the conditions on the degrees of $C_n(x)$ and $D_n(x)$, combined with the degrees of $\Phi$ and $\Psi$ not being greater than $1$, imply that \eqref{conditions on the degrees of the type I polynomials} holds.
\end{proof}

\begin{remark}\label{rem: type I polys general} As a straightforward consequence of Proposition \ref{differentiation formula for type I polynomials - general result},
 the type I multiple orthogonal polynomials $(A_n(x) , B_n(x))$ and $(C_n(x), D_n(x))$ for $\dis\ \overline{w}(x)$ and for $\dis\  x\,\Phi(x)\overline{w}(x)$, respectively, are related by 
 \[
\left(\begin{array}{c}
	A_{n+1}(x) \\
	B_{n+1}(x)
\end{array}\right)
= x \Phi(x)^{t}  \left(\begin{array}{c}
	C_{n}^{\ \prime}(x) \\
	D_{n}^{\ \prime}(x)
\end{array}\right)
- \Psi(x)^{t} \left(\begin{array}{c}
	C_{n}(x) \\
	D_{n}(x)
\end{array}\right),\quad \text{for all} \quad n\geq 0, 
 \]
 where $\Phi^{t}$ and $\Psi^{t}$ are the transpose of the matrices given in Proposition \ref{differentiation formula for type II polynomials - general result}. 

\end{remark}

Combining Propositions \ref{differentiation formula for type II polynomials - general result} and \ref{differentiation formula for type I polynomials - general result} with Theorem \ref{canonical matrix differential equation satisfied by the weights - theorem}, we deduce differential and difference properties for type I and type II multiple polynomials, which are described in the following result.

\begin{theorem}
\label{differentiation formulas for the multiple orthogonal polynomials - theorem}
Let $a,b,c\in\R$ such that $a>-1$, $b>-1$ and $c>\max\{0,a-b\}$ and $d\in\{0,1\}$.
 \linebreak If $\ \dis\polyseq[P_n^\parameter{d}(x;a,b,c)]$ is the monic $2$-orthogonal polynomial sequence and $\dis Q_n^\parameter{d}(x;a,b,c)$, with $n\in\N$, the type I function for the index of length $n$ on the step-line with respect to $\dis\Wvec^\parameter{d}(x;a,b,c)$, 
then
\begin{align}
\label{differentiation formula for type II polynomials}
\DiffOp\left(P_{n+1}^\parameter{d}(x;a,b,c)\right)=(n+1)P_n^\parameter{1-d}(x;a+1,b+1,c+d),
\end{align}
and 
\begin{align}
\label{differentiation formula for type I polynomials}
\DiffOp\left(Q_n^\parameter{1-d}(x;a+1,b+1,c+d)\right)
=-n\,Q_{n+1}^\parameter{d}(x;a,b,c). 
\end{align}
\end{theorem}

\begin{proof}
Let $\dis\Phi^\parameter{d}(x;a,b,c)$  be defined as in Theorem \ref{canonical matrix differential equation satisfied by the weights - theorem}.

Proposition \ref{differentiation formula for type II polynomials - general result} ensures that  $\polyseq[\frac{1}{n+1}\DiffOp\left(P_{n+1}^\parameter{d}(x;a,b,c)\right)]$ is $2$-orthogonal with respect to the vector of weights $\dis x\Phi^\parameter{d}(x;a,b,c)\Wvec^\parameter{d}(x;a,b,c)$.
Besides, from Proposition \ref{differentiation formula for type I polynomials - general result}, we know that, if $\tilde{Q}_n^\parameter{d}(x;a,b,c)$ is the type I function for the index of length $n$ on the step-line with respect to $\dis x\Phi^\parameter{d}(x;a,b,c)\Wvec^\parameter{d}(x;a,b,c)$, then $\dis -\frac{1}{n}\DiffOp\left(\tilde{Q}_n^\parameter{d}(x;a,b,c)\right)$ is the type I function for the index of length $n+1$ on the step line with respect to the vector of weights $\dis\Wvec^\parameter{d}(x;a,b,c)$.

By virtue of \eqref{shift on the parameters} 
we conclude that \eqref{differentiation formula for type II polynomials} holds, and also that $\dis\tilde{Q}_n^\parameter{d}(x;a,b,c)=Q_n^\parameter{1-d}(x;a+1,b+1,c+d)$ is valid for any $n\in\N$. Hence, \eqref{differentiation formula for type I polynomials} also holds.
\end{proof}

The differentiable properties described in Theorem \ref{differentiation formulas for the multiple orthogonal polynomials - theorem} are the main pillars for further characterisation of the multiple orthogonal polynomials under analysis. These intrinsic properties resemble those found within the context of the very classical standard orthogonal polynomials. 

\subsection{Type I multiple orthogonal polynomials}
\label{Rodrigues-type formula for type I multiple orthogonal polynomials}

Let us revisit Proposition \ref{differentiation formula for type I polynomials - general result} and Remark  \ref{rem: type I polys general} for the case where the vector of  weights $w$ is replaced by  $\dis\Wvec^\parameter{d}(x;a,b,c)$ defined by \eqref{Wvec}. We recall \eqref{shift on the parameters} in Theorem \ref{canonical matrix differential equation satisfied by the weights - theorem} and \eqref{differentiation formula for type I polynomials} in Theorem \ref{differentiation formulas for the multiple orthogonal polynomials - theorem} to conclude that if $(A_n^\parameter{d}(x;a,b,c),B_n^\parameter{d}(x;a,b,c))$ is the vector of type I multiple orthogonal polynomials on the step line for  $\dis\Wvec^\parameter{d}(x;a,b,c)$, then 
 \[\begin{multlined}
\left(\begin{array}{c}
	A_{n+1}^\parameter{d}(x;a,b,c) \\[0.2cm]
	B_{n+1}^\parameter{d}(x;a,b,c)
\end{array}\right) 
= \left( x \Phi(x)^{t}\DiffOpFunction{} 
- \Psi(x)^{t} 
 \right) \left(\begin{array}{c}
	A_{n}^\parameter{1-d}(x;a+1,b+1,c+d) \\[0.2cm]
	B_{n}^\parameter{1-d}(x;a+1,b+1,c+d)
\end{array}\right),
\quad \text{for all} \quad n\geq 0.
 \end{multlined}
 \]

The type I multiple orthogonal functions on the step line can be generated by concatenated differentiation of the weight function or, in other words, via a Rodrigues-type formula.

\begin{theorem}
\label{Rodrigues formula for the type I polynomials - theorem}
Let $a,b,c\in\R$ such that $a>-1$, $b>-1$ and $c>\max\{0,a-b\}$ and, for each $n\geq 1$, let $Q_n^\parameter{d}(x;a,b,c)$, $d\in\{0,1\}$, be the type I function for the index of length $n$ on the step line with respect to $\dis\Wvec^\parameter{d}(x;a,b,c)$.
Then
\begin{equation}
\label{Rodrigues formula for the type I polynomials}
Q_n^\parameter{d}(x;a,b,c)
=\frac{(-1)^{n-1}}{(n-1)!}\DiffOpHigherOrder{n-1}
\left(\W\left(x;a+n-1,b+n-1,c+\floor{\frac{n+d}{2}}\right)\right)
\end{equation}
\end{theorem}

\begin{proof} We proceed by induction. 
For $n=1$, the relation \eqref{Rodrigues formula for the type I polynomials} trivially holds, because it reads as $\dis Q_1^\parameter{d}(x;a,b,c)=\W(x;a,b,c+d)$ and, recalling \eqref{moments}, we have 
\begin{equation*}
\int_{0}^{\infty}\W(x;a,b,c+d)\dx=1. 
\end{equation*}

Observe that $\dis\W(x;a+1,b+1,c+1)=\frac{c+b+2}{(a+1)(b+1)}\,x\,\W(x;a,b,c+1)$ and use \eqref{W0'(c+1) as combination of W0(c) and W0(c+1)} to write 
\begin{equation*}
-\DiffOp\left(\W(x;a+1,b+1,c+1)\right)=
\frac{(c+b+1)(c+b+2)}{(a+1)(b+1)}\left(\W(x;a,b,c)-\W(x;a,b,c+1)\right).
\end{equation*}
As a result, relation \eqref{Rodrigues formula for the type I polynomials} also holds for $n=2$, because, based on \eqref{moments}, we can deduce 
\begin{align*}
&\int_{0}^{\infty}\DiffOp\left(\W(x;a+1,b+1,c+1)\right)\dx
	\\
	&\qquad\qquad
=\frac{(c+b+1)(c+b+2)}{(a+1)(b+1)}\left(\int_{0}^{\infty}\W(x;a,b,c+1)-\int_{0}^{\infty}\W(x;a,b,c)\dx\right)\dx
	=0, 
\end{align*}
as well as 
\begin{align*}
&- \int_{0}^{\infty} x\DiffOp\left(\W(x;a+1,b+1,c+1)\right)\dx 
\\
&\qquad\qquad =\frac{(c+b+1)(c+b+2)}{(a+1)(b+1)}\left(\int_{0}^{\infty}x\W(x;a,b,c)\dx-\int_{0}^{\infty}x\W(x;a,b,c+1)\right)\dx 
\\
&\qquad\qquad 
=\frac{(c+b+1)(c+b+2)}{(a+1)(b+1)}\left(\frac{(a+1)(b+1)}{c+b+1}-\frac{(a+1)(b+1)}{c+b+2}\right)\dx
=1.
\end{align*}

Using \eqref{differentiation formula for type I polynomials} and then evoking the assumption that \eqref{Rodrigues formula for the type I polynomials} holds for a fixed $n\geq 2$, we obtain
\begin{align*}
 Q_{n+1}^\parameter{d}(x;a,b,c)
&=-\frac{1}{n}\DiffOp\left(Q_{n}^\parameter{1-d}(x;a+1,b+1,c+d)\right) \\
&=\frac{(-1)^n}{n!}\DiffOpHigherOrder{n}\left(\W\left(x;a+n,b+n,c+\floor{\frac{n+1+d}{2}}\right)\right).
\end{align*}
If we equate the first and latter members, we readily see that \eqref{Rodrigues formula for the type I polynomials} holds for $n+1$ and, as a result, we can state that it holds for all $n\in\Z^+$ by induction.
\end{proof}

From this point forth the focus will be on the type II multiple orthogonal polynomials. 

\section{Characterisation of the type II multiple orthogonal polynomials}
\label{Type II multiple orthogonal polynomials}

One of the defining properties of the four families of the very classical orthogonal polynomials (of Hermite, Laguerre, Bessel and Jacobi) is that the sequence of their derivatives is also orthogonal (ie, they satisfy Hahn's property). Together, these four families share a number of other properties. We highlight two of them. They are polynomial solutions to a second order linear differential equation with polynomial coefficients (the so called Bochner's differential equation). Their orthogonality weight functions are solutions to a first order homogeneous linear differential equation with polynomial coefficients (commonly referred to as the Pearson equation). 

The type II multiple polynomials on the step line $\polyseq[P_n^\parameter{d}(x)]$ orthogonal for  $\dis\Wvec^\parameter{d}(x;a,b,c)$ in \eqref{Wvec} also satisfy a number of properties that resemble those found among the classical polynomials. The relation between $\polyseq[P_n^\parameter{d}(x;a,b,c)]$  and its sequence of derivatives given in \eqref{differentiation formula for type II polynomials} clearly shows that \linebreak $\polyseq[P_n^\parameter{d}(x;a,b,c)]$ satisfy the Hahn's property. Additionally, we show in Section \ref{subsec: Diff eq} that they are solution to a third order linear differential equation with polynomial coefficients and we have described the vector of weight functions to be a solution to a first order homogeneous linear matrix differential equation resembling a matrix version of the Pearson equation. Therefore, it makes all sense to perceive these polynomials as (Hahn)-classical polynomials in the context of multiple orthogonality. (However, within this context, it is worth to note that in the literature there are other notions of "classical".) As a matter of fact, such properties are also shared by other Hahn-classical $2$-orthogonal polynomials. In \cite{AnaandWalter} the equivalence between these three properties was proved for the threefold symmetric case.

\subsection{Explicit expression }

Based on the differential relation  \eqref{differentiation formula for type II polynomials}, we deduce an explicit expression for the type II multiple orthogonal polynomials on the step line$\polyseq[P_n^\parameter{d}(x;a,b,c)]$  as a generalised hypergeometric function. 

\begin{theorem}
\label{explicit formulas for the 2-orthogonal polynomials}
Let $a,b,c\in\R$ such that $a>-1$, $b>-1$ and $c>\max\{0,a-b\}$ and let $\dis\polyseq[P_n^\parameter{d}(x):=P_n^\parameter{d}(x;a,b,c)]$, $d\in\{0,1\}$, be the monic $2$-orthogonal polynomial sequence with respect to $\dis\Wvec^\parameter{d}(x;a,b,c)$.
Then
\begin{subequations}
\begin{align}
\label{explicit formula for the 2-orthogonal polynomials as a 2F2}
P_n^\parameter{d}(x)
=\frac{(-1)^n\pochhammer{a+1}\pochhammer{b+1}}{\pochhammer{c+b+1+\floor{\frac{n+d}{2}}}}
\,\twoFtwo{-n,c+b+1+\floor{\frac{n+d}{2}}}{a+1,b+1}
\end{align}
or, equivalently,
\begin{align}
\label{expansion of the $2$-orthogonal polynomials over the canonical basis}
P_n^\parameter{d}(x)
=\sum_{j=0}^{n} \tau_{n,j}^\parameter{d}
\ x^j
\quad \text{with} \quad 
\tau_{n,j}^\parameter{d}=(-1)^{n-j}\binom{n}{j}
\frac{\pochhammer[n-j]{a+1+j}\pochhammer[n-j]{b+1+j}}
{\pochhammer[n-j]{c+b+1+\floor{\frac{n+d}{2}}+j}}.
\end{align}
\end{subequations}
\end{theorem}

To prove this theorem we need to check that $\polyseq[P_n^\parameter{d}(x)]$ in \eqref{expansion of the $2$-orthogonal polynomials over the canonical basis} satisfies the $2$-orthogonality conditions with respect to $\dis\Wvec^\parameter{d}(x;a,b,c)$, which are:  
\begin{subequations}
\begin{align}
\label{ortho Pn 0}
\int_{0}^{\infty}x^kP_n^\parameter{d}(x)\W(x;a,b,c+d)\dx
=\begin{cases}
		0, &\text{ if } n\geq 2k+1,\\
N_n\neq 0, &\text{ if } n=2k,
\end{cases}
\end{align}
and
\begin{align}
\label{ortho Pn 1}
\int_{0}^{\infty}x^kP_n^\parameter{d}(x)\W(x;a,b,c+1-d)\dx
=\begin{cases}
		0, &\text{ if } n\geq 2k+2,\\
N_n\neq 0, &\text{ if } n=2k+1, 
\end{cases}
\end{align}
where it is understood that $N_n:=N_n^{[d]}(a,b,c)\neq 0$ for all $n\in\mathbb{N}$. 

Actually, as we are dealing with a Nikishin system, the existence of a $2$-orthogonal polynomial sequence with respect to $\dis\Wvec^\parameter{d}(x;a,b,c)$ is guaranteed. By virtue of the known differential formula for a generalised hypergeometric series \cite[Eq.~16.3.1]{DLMF}, it is rather straightforward to show that the polynomials given by \eqref{explicit formula for the 2-orthogonal polynomials as a 2F2} satisfy the differential property \eqref{differentiation formula for type II polynomials} stated in Proposition \ref{differentiation formula for type II polynomials - general result}. A property that a $2$-orthogonal polynomial sequence with respect to $\dis\Wvec^\parameter{d}(x;a,b,c)$ must satisfy. Therefore, it would be sufficient to check the orthogonality conditions \eqref{ortho Pn 0}-\eqref{ortho Pn 1} when $k=0$ to then prove the result by induction on $n\in\N$ (the degree of the polynomials).

However, we opt for checking that the polynomials $P_n^\parameter{d}(x)$ in \eqref{explicit formula for the 2-orthogonal polynomials as a 2F2} satisfy all the orthogonality conditions \eqref{ortho Pn 0}-\eqref{ortho Pn 1}. On the one hand, this process enables us to show directly that the polynomials in  \eqref{explicit formula for the 2-orthogonal polynomials as a 2F2} are indeed $2$-orthogonal with respect to $\dis\Wvec^\parameter{d}(x;a,b,c)$ without arguing with the Nikishin property. On the other hand, it provides a method to derive the following explicit expressions for $N_n$ in  \eqref{ortho Pn 0}-\eqref{ortho Pn 1} 
\begin{align}\label{N2k}
	N_{2k}^{[d]}(a,b,c) 
	=\begin{cases}
   	  	\dis \frac{(2k)!\pochhammer[2k]{a+1}\pochhammer[2k]{b+1}}
    {\pochhammer[3k]{c+b+1}}  & \text{if}\quad d=0, 
    			\\[0.5cm]
	  \dis\frac{(2k)!\pochhammer[2k]{a+1}\pochhammer[2k]{b+1}\pochhammer[k]{c+b-a+1}\pochhammer[k]{c+1}
	  			\pochhammer[k]{c+b+1}}
    {\pochhammer[3k]{c+b+1}\pochhammer[3k]{c+b+2}}
	 &  \text{if}\quad d=1, 
    \end{cases} 
\end{align}
and 
\begin{align}\label{N2k plus 1}
	N_{2k+1}^{[d]}(a,b,c) 
	= \begin{cases}
	\dis - \frac{(2k+1)!\pochhammer[2k+1]{a+1}\pochhammer[2k+1]{b+1}\pochhammer[k]{c+b-a+1}\pochhammer[k]{c+1}\pochhammer[k]{c+b+1}}
    {\pochhammer[3k+1]{c+b+1}\pochhammer[3k+1]{c+b+2}}
			 &\text{if}\ d=0,
	\\[0.5cm]
	\dis \frac{(2k+1)!\pochhammer[2k+1]{a+1}\pochhammer[2k+1]{b+1} }
    {\pochhammer[3k+2]{c+b+1}}
	&\text{if}\ d=1.
	\end{cases} 
\end{align}
\end{subequations}
Note that $N_{2k+j}^{[d]}(a,b,c) $ are all positive except when $j=1$ and $d=0$.

The explicit expression for $N_n$ readily gives an explicit expression for the nonzero $\gamma$-coefficients in the third order recurrence relation \eqref{recurrence relation for a 2-OPS} satisfied by these polynomials. Such relation is discussed in Section  \ref{Recurrence relation}. All in all, we present two alternative proofs for Theorem \ref{explicit formulas for the 2-orthogonal polynomials}.

In order to do so, we need following technical lemma.

\begin{lemma}
\label{explicit formulas for hypergeometric series}
Let $n$, $p$, and $m_1,\cdots,m_p$ be positive integers such that $\dis m:=\sum_{i=1}^{p}m_i\leq n$ and $\beta,f_1,\cdots,f_p$ be complex numbers with positive real part.
Then
\begin{align}
\label{hypergeometric formula 1}
\pochhammer[m_1]{f_1}\cdots\pochhammer[m_p]{f_p}
\Hypergeometric[1]{p+1}{p}{-n,f_1+m_1,\cdots,f_p+m_p}{f_1,\cdots,f_p}
=\pochhammer{-m}
=\begin{cases}
(-1)^n n! & \text{ if } m=n,\\
		0 & \text{ if } m<n.
\end{cases}
\end{align}
and
\begin{align}
\label{Minton's formula}
\Hypergeometric[1]{p+2}{p+1}{-n,\beta,f_1+m_1,\cdots,f_p+m_p}{\beta+1,f_1,\cdots,f_p}=
\frac{n!\pochhammer[m_1]{f_1-\beta}\cdots\pochhammer[m_p]{f_p-\beta}}
{\pochhammer{\beta+1}\pochhammer[m_1]{f_1}\cdots\pochhammer[m_p]{f_p}}.
\end{align}
In particular,
\begin{align}
\label{hypergeometric formula Marjolein+Walter}
\sum_{j=0}^{n}(-1)^{n-j}\binom{n}{j}\pochhammer[m]{j+1}
=(-1)^m m!\,\Hypergeometric[1]{2}{1}{-n,m+1}{1}
=\begin{cases}
n! & \text{ if } m=n,\\
 0 & \text{ if } m<n.
\end{cases}
\end{align}
	
\end{lemma}

\begin{proof}
Formula \eqref{Minton's formula} was deduced by Minton in \cite{Minton} (see also \cite{Karp}).
A proof of \eqref{hypergeometric formula Marjolein+Walter} can be found in \cite{MarjoleinandWalter}.
Thus, we only need to prove \eqref{hypergeometric formula 1}.
By definition of the generalised hypergeometric series \eqref{generalised hypergeometric function},
\begin{align*}
\Hypergeometric[1]{p+1}{p}{-n,f_1+m_1,\cdots,f_p+m_p}{f_1,\cdots,f_p}
=\sum_{j=0}^{n}\frac{\pochhammer[j]{-n}\pochhammer[j]{f_1+m_1}\cdots\pochhammer[j]{f_p+m_p}}
{j!\pochhammer[j]{f_1}\cdots\pochhammer[j]{f_p}}.
\end{align*}
We have $\dis\frac{\pochhammer[j]{-n}}{j!}=(-1)^j\binom{n}{j}$ and
$\dis\frac{\pochhammer[j]{f_i+m_i}}{\pochhammer[j]{f_i}}
=\frac{\Gamma\left(f_i+m_i+j\right)\Gamma\left(f_i\right)}
{\Gamma\left(f_i+m_i\right)\Gamma\left(f_i+j\right)}
=\frac{\pochhammer[m_i]{f_i+j}}{\pochhammer[m_i]{f_i}}$, 
for each $1\leq i\leq p$, so that 
\begin{align*}
\pochhammer[m_1]{f_1}\cdots\pochhammer[m_p]{f_p}
\Hypergeometric[1]{p+1}{p}{-n,f_1+m_1,\cdots,f_p+m_p}{f_1,\cdots,f_p}
=\sum_{j=0}^{n}(-1)^j\binom{n}{j}
\pochhammer[m_1]{f_1+j}\cdots\pochhammer[m_p]{f_p+j}. 
\end{align*}
Observe that  $\dis\pochhammer[m_1]{f_1+j}\cdots\pochhammer[m_p]{f_p+j}$ is a monic polynomial of degree $m=m_1+\cdots+m_p$ on the variable $j$. Therefore,  using \eqref{hypergeometric formula Marjolein+Walter}, this sum equals $0$ for any $m<n$, and it equals $(-1)^nn!$ when $m=n$.  
\end{proof}

\begin{proof}[Proof of Theorem \ref{explicit formulas for the 2-orthogonal polynomials}.] 
We evaluate the left hand side of \eqref{ortho Pn 0} for any $k,n\in\N$, by using the expression for the moments \eqref{moments} and the polynomial expansion \eqref{expansion of the $2$-orthogonal polynomials over the canonical basis}. This successively gives 
\begin{align*}
&\int_{0}^{\infty}x^kP_n^\parameter{d}(x;a,b,c)\W(x;a,b,c+d)\dx 
	=\sum_{j=0}^{n}\tau_{n,j}^\parameter{d}\moment[j+k](a,b,c+d) \\
&    =\sum_{j=0}^{n}
    (-1)^{n-j}\binom{n}{j}
    \frac{\pochhammer[n-j]{a+1+j}\pochhammer[n-j]{b+1+j}}
    {\pochhammer[n-j]{c+b+1+\floor{\frac{n+d}{2}}+j}}
    \frac{\pochhammer[k+j]{a+1}\pochhammer[k+j]{b+1}}
    {\pochhammer[k+j]{c+b+1+d}} \\
&=\frac{(-1)^n\pochhammer{a+1}\pochhammer{b+1}\pochhammer[k]{a+1}\pochhammer[k]{b+1}}
    {\pochhammer{c+b+1+\floor{\frac{n+d}{2}}}\pochhammer[k]{c+b+1+d}}
    \,\sum_{j=0}^{n}
    \frac{\pochhammer[j]{-n}\pochhammer[j]{a+1+k}\pochhammer[j]{b+1+k}\pochhammer[j]{c+b+1+\floor{\frac{n+d}{2}}}}
    {j!\pochhammer[j]{a+1}\pochhammer[j]{b+1}\pochhammer[j]{c+b+1+d+k}} \\
&=\frac{(-1)^n\pochhammer{a+1}\pochhammer{b+1}\pochhammer[k]{a+1}\pochhammer[k]{b+1}}
    {\pochhammer{c+b+1+\floor{\frac{n+d}{2}}}\pochhammer[k]{c+b+1+d}}
    \,\Hypergeometric[1]{4}{3}{-n,a+1+k,b+1+k,c+b+1+\floor{\frac{n+d}{2}}}{a+1,b+1,c+b+1+d+k}.
\end{align*}
Based on \eqref{hypergeometric formula 1} in Lemma \ref{explicit formulas for hypergeometric series}, we obtain 
\[
	\Hypergeometric[1]{4}{3}{-n,a+1+k,b+1+k,c+b+1+\floor{\frac{n+d}{2}}}{a+1,b+1,c+b+1+d+k} = 0 , \quad \text{for any }\quad n\geq 2k+1, 
\]
so that 
\begin{subequations} 
\begin{equation}	\label{pf ortho cond 1 zero}
	\int_{0}^{\infty}x^kP_n^\parameter{d}(x)\W(x;a,b,c+d)\dx = 0,  \quad \text{for any }\quad n\geq 2k+1. 
\end{equation}
Besides, 
\begin{align*}
&	\int_{0}^{\infty}x^kP_{2k}^\parameter{d}(x;a,b,c)\W(x;a,b,c+d)\dx 
	\\
&	= \frac{\pochhammer[2k]{a+1}\pochhammer[2k]{b+1}\pochhammer[k]{a+1}\pochhammer[k]{b+1}}
    {\pochhammer[2k]{c+b+1+k+\floor{\frac{d}{2}}}\pochhammer[k]{c+b+1+d}}
    \,\Hypergeometric[1]{4}{3}{-2k,a+1+k,b+1+k,c+b+1+k+\floor{\frac{d}{2}}}{a+1,b+1,c+b+1+d+k} 
    	\notag \\
& 	= \frac{\pochhammer[2k]{a+1}\pochhammer[2k]{b+1}\pochhammer[k]{a+1}\pochhammer[k]{b+1}}
    {\pochhammer[2k]{c+b+1+k}\pochhammer[k]{c+b+1+d}}
    \,\Hypergeometric[1]{4}{3}{-2k,a+1+k,b+1+k,c+b+1+k}{a+1,b+1,c+b+1+d+k} 
    	\notag 
\end{align*}
For the $d=0$ the latter hypergeometric series simplifies to a ${}_3F_2$, which on account of the identity \eqref{hypergeometric formula 1}, it can be evaluated to 
\begin{align*}
	&
	\Hypergeometric[1]{3}{2}{-2k,a+1+k,b+1+k}{a+1,b+1} 
	= \frac{(2k)!}{\pochhammer[k]{a+1}\pochhammer[k]{b+1}}
	 \quad \text{if}\quad d=0, 
\end{align*}
whilst for the case where $d=1$ we have to use \eqref{Minton's formula} to get 
\begin{align*}
	&\Hypergeometric[1]{4}{3}{-2k,a+1+k,b+1+k,c+b+1+k}{a+1,b+1,c+b+2+k} 
    =\frac{(2k)!\pochhammer[k]{c+b-a+1}\pochhammer[k]{c+1}}
    {\pochhammer[2k]{c+b+2+k}\pochhammer[k]{a+1}\pochhammer[k]{b+1}},
	 \quad \text{if}\quad d=1.
\end{align*}
As a result, we have 
\begin{align}
& 	\label{pf ortho cond 1}
	\int_{0}^{\infty}x^kP_{2k}^\parameter{d}(x;a,b,c)\W(x;a,b,c+d)\dx 
 = \begin{cases}
   	  	\frac{(2k)!\pochhammer[2k]{a+1}\pochhammer[2k]{b+1}}
    {\pochhammer[3k]{c+b+1}}  & \text{if}\quad d=0, 
    			\\[0.5cm]
	  \frac{(2k)!\pochhammer[2k]{a+1}\pochhammer[2k]{b+1}\pochhammer[k]{c+b-a+1}\pochhammer[k]{c+1}}
    {\pochhammer[2k]{c+b+1+k}\pochhammer[3k]{c+b+2}}
	 &  \text{if}\quad d=1. 
    \end{cases} 
\end{align}
Hence,  \eqref{pf ortho cond 1 zero} together with \eqref{pf ortho cond 1} lead to \eqref{ortho Pn 0} and \eqref{N2k}. 
\end{subequations}

Using similar arguments as before, we recall \eqref{moments}, to write the left hand side of \eqref{ortho Pn 1} as 
\begin{subequations}
\begin{align*}
&\int_{0}^{\infty}x^kP_n^\parameter{d}(x;a,b,c)\W(x;a,b,c+1-d)\dx 
	\ = \ \sum_{j=0}^{n}\tau_{n,j}^\parameter{d}\moment[j+k](a,b,c+1-d) \\
&\qquad=\frac{(-1)^n\pochhammer{a+1}\pochhammer{b+1}\pochhammer[k]{a+1}\pochhammer[k]{b+1}}
    {\pochhammer{c+b+1+\floor{\frac{n+d}{2}}}\pochhammer[k]{c+b+2-d}}
    \,\Hypergeometric[1]{4}{3}{-n,a+1+k,b+1+k,c+b+1+\floor{\frac{n+d}{2}}}{a+1,b+1,c+b+2-d+k}.
\end{align*}
Identity \eqref{hypergeometric formula 1} in Lemma \ref{explicit formulas for hypergeometric series} implies 
\[
	\Hypergeometric[1]{4}{3}{-n,a+1+k,b+1+k,c+b+1+\floor{\frac{n+d}{2}}}{a+1,b+1,c+b+2-d+k} = 0 
	\quad\text{for}\quad n\geq 2k+2, 
\]
which means that 
\begin{equation}	\label{pf ortho cond 2 zero}
	\int_{0}^{\infty}x^kP_n^\parameter{d}(x)\W(x;a,b,c+1-d)\dx = 0,  \quad \text{for any }\quad n\geq 2k+2. 
\end{equation}
When $n=2k+1$, the left hand side of \eqref{ortho Pn 1} becomes as 
\begin{align*}
&\int_{0}^{\infty}x^kP_{2k+1}^\parameter{d}(x;a,b,c)\W(x;a,b,c+1-d)\dx 
	\\
&=- \frac{\pochhammer[2k+1]{a+1}\pochhammer[2k+1]{b+1}\pochhammer[k]{a+1}\pochhammer[k]{b+1}}
    {\pochhammer[2k+1]{c+b+1+k+d}\pochhammer[k]{c+b+2-d}}
    \,\Hypergeometric[1]{4}{3}{-(2k+1),a+1+k,b+1+k,c+b+1+k+d}{a+1,b+1,c+b+2-d+k}.
\end{align*}
In order to evaluate the terminating hypergeometric series in the latter expression, we use 
\eqref{Minton's formula} for the case where $d=0$ and we use \eqref{hypergeometric formula 1} when $d=1$. 
This gives 
\begin{align*}
	& \,\Hypergeometric[1]{4}{3}{-(2k+1),a+1+k,b+1+k,c+b+1+k+d}{a+1,b+1,c+b+2-d+k} 
	= \begin{cases}
	\frac{(2k+1)!\pochhammer[k]{c+b-a+1}\pochhammer[k]{c+1}}
{\pochhammer[2k+1]{c+b+2+k}\pochhammer[k]{a+1}\pochhammer[k]{b+1}} &\text{for}\ d=0,
	\\[0.4cm]
	-\frac{(2k+1)!}{\pochhammer[k]{a+1}\pochhammer[k]{b+1} (c+b+k+1)} &\text{for}\ d=1, 
	\end{cases}
\end{align*}
so that 
\begin{align}
&\int_{0}^{\infty}x^kP_{2k+1}^\parameter{d}(x;a,b,c)\W(x;a,b,c+1-d)\dx 
	\label{pf ortho cond 2}\\
&\notag\qquad 
	= \begin{cases}
	\dis - \frac{(2k+1)!\pochhammer[2k+1]{a+1}\pochhammer[2k+1]{b+1}\pochhammer[k]{c+b-a+1}\pochhammer[k]{c+1}\pochhammer[k]{c+b+1}}
    {\pochhammer[3k+1]{c+b+1}\pochhammer[3k+1]{c+b+2}}
			 &\text{for}\ d=0,
	\\[0.5cm]
	\dis \frac{(2k+1)!\pochhammer[2k+1]{a+1}\pochhammer[2k+1]{b+1} }
    {\pochhammer[3k+2]{c+b+1}}
	&\text{for}\ d=1. 
	\end{cases}
\end{align}
The latter identity combined with \eqref{pf ortho cond 2 zero} ensure that \eqref{ortho Pn 1} and \eqref{N2k plus 1} hold for any $k$ and $n$. 
\end{subequations}
\end{proof}

\subsection{Differential equation}\label{subsec: Diff eq}

The type II multiple orthogonal polynomials of hypergeometric  type described in \eqref{explicit formula for the 2-orthogonal polynomials as a 2F2} are solutions to a third order differential equation. The structure of this differential equation resembles the structure of differential equations satisfied by other $2$-orthogonal polynomials satisfying the Hahn property, that is, Hahn-classical polynomials. Examples of such polynomials can be found for instance in  \cite{BenCheikhandDouak, DouakandMaroniClassiquesDeDimensionDeux,SemyonandWalter} among other works. 

In a way this differential equation can be seen as a Bochner type differential equation satisfied by all the classical (standardly) orthogonal polynomials. 

\begin{theorem}
\label{thm: diff eq}
Let $a,b,c\in\R$ such that $a>-1$, $b>-1$ and $c>\max\{0,a-b\}$ and let $\dis\polyseq[P_n^\parameter{d}(x):=P_n^\parameter{d}(x;a,b,c)]$, $d\in\{0,1\}$, be the monic $2$-orthogonal polynomial sequence with respect to $\dis\Wvec^\parameter{d}(x;a,b,c)$. 
Then
\begin{align}
\label{differential equation satisfied by the 2-orthogonal polynomials}
 x^2\DiffOpHigherOrder{3}\left(P_n^\parameter{d}(x)\right)
-x\varphi(x)\DiffOpHigherOrder{2}\left(P_n^\parameter{d}(x)\right)
+\psi_n^\parameter{d}(x)\DiffOp\left(P_n^\parameter{d}(x)\right)
+n\epsilon_n^\parameter{d}P_n^\parameter{d}(x)=0,
\end{align}
where $\dis\varphi(x)=x-(a+b+3)$,
$\dis\psi_n^\parameter{d}(x)=\left(\floor{\frac{n+1-d}{2}}-(c+b+2)\right)x+(a+1)(b+1)$
and \[\dis\epsilon_n^\parameter{d}=c+b+1+\floor{\frac{n+d}{2}}.\]
\end{theorem}

\begin{proof}
Combining the explicit formula for the polynomials given by \eqref{explicit formula for the 2-orthogonal polynomials as a 2F2} and the generalised hypergeometric differential equation \cite[Eq.~16.8.3]{DLMF}, we obtain
\begin{equation}\label{diff eq v2}
\left[\DiffOp\left(x\,\DiffOp+a\right)\left(x\,\DiffOp+b\right)\right]P_n^\parameter{d}(x)
=\left[\left(x\,\DiffOp+\epsilon_n^\parameter{d}\right)\left(x\,\DiffOp-n\right)\right]P_n^\parameter{d}(x).
\end{equation}

Moreover, we have
\begin{align*}
\left[\left(x\,\DiffOp+a\right)\left(x\,\DiffOp+b\right)\right]\left(P_n^\parameter{d}(x)\right) 
=x^2\DiffOpHigherOrder{2}\left(P_n^\parameter{d}(x)\right)
+(a+b+1)x\,\DiffOp\left(P_n^\parameter{d}(x)\right)
+abP_n^\parameter{d}(x)
\end{align*}
and, observing that $\dis\left(\epsilon_n^\parameter{d}-n+1\right)=(c+b+2)-\floor{\frac{n+1-d}{2}}$,
\begin{equation*}\begin{multlined}
\left[\left(x\,\DiffOp+\epsilon_n^\parameter{d}\right)\left(x\,\DiffOp-n\right)\right]\left(P_n^\parameter{d}(x)\right)\\
=x^2\DiffOpHigherOrder{2}\left(P_n^\parameter{d}(x)\right)
-\left(\floor{\frac{n+1-d}{2}}-(c+b+2)\right)x\,\DiffOp\left(P_n^\parameter{d}(x)\right)
-n\epsilon_n^\parameter{d}P_n(x).
\end{multlined}
\end{equation*}

Furthermore, differentiating our expression for 
$\dis\left[\left(x\,\DiffOp+a\right)\left(x\,\DiffOp+b\right)\right]\left(P_n^\parameter{d}(x)\right)$,
we obtain
\begin{equation*}
\begin{multlined}
\label{differential equation LHS2}
\left[\DiffOp\left(x\,\DiffOp+a\right)\left(x\,\DiffOp+b\right)\right]P_n(x)\\
=x^2\,\DiffOpHigherOrder{3}\left(P_n^\parameter{d}(x)\right)
+(a+b+3)x\,\DiffOpHigherOrder{2}\left(P_n^\parameter{d}(x)\right)
+(a+1)(b+1)\DiffOp\left(P_n^\parameter{d}(x)\right).
\end{multlined}
\end{equation*}

Finally, combining the former and the latter expressions, we derive the differential equation \eqref{differential equation satisfied by the 2-orthogonal polynomials}.
\end{proof}

The alternative representation \eqref{diff eq v2} for the differential equation \eqref{differential equation satisfied by the 2-orthogonal polynomials} highlights some symmetrical pro\-per\-ties of these polynomials that may be worth to explore.

\subsection{Recurrence relation}
\label{Recurrence relation}

One of the main features of the multiple orthogonal polynomials of type II on the step line (or the also referred to as {$2$-orthogonal polynomials}) is the third order recurrence relation. As a $2$-orthogonal sequence, the hypergeometric type polynomials defined by \eqref{explicit formula for the 2-orthogonal polynomials as a 2F2} necessarily satisfy a recurrence relation of the type 
\begin{align}
\label{recurrence relation satisfied by the 2-orthogonal polynomials}
P_{n+1}^\parameter{d}(x)
=\left(x-\beta_n^\parameter{d}\right)P_n^\parameter{d}(x)
-\alpha_n^\parameter{d} P_{n-1}^\parameter{d}(x)
-\gamma_{n-1}^\parameter{d} P_{n-2}^\parameter{d}(x).
\end{align}
The key point here is to obtain explicit expressions for the recurrence coefficients triplet $(\beta_n^\parameter{d},\alpha_n^\parameter{d},\gamma_n^\parameter{d})$ in \eqref{recurrence relation satisfied by the 2-orthogonal polynomials}. For simplification, we have written $\beta_n^\parameter{d}:=\beta_n^\parameter{d}(a,b,c),\ \alpha_n^\parameter{d}:=\alpha_n^\parameter{d}(a,b,c)$ and $\gamma_n^\parameter{d}:=\gamma_n^\parameter{d}(a,b,c)$.

Their expressions can be derived through the explicit expression given in \eqref{explicit formula for the 2-orthogonal polynomials as a 2F2} or in \eqref{expansion of the $2$-orthogonal polynomials over the canonical basis}. For that, in the recurrence relation \eqref{recurrence relation satisfied by the 2-orthogonal polynomials} we replace the polynomials $P_{n+1-j}^\parameter{d}(x)$ (with $j=0,1,2,3$) by their corresponding expansion expression  \eqref{expansion of the $2$-orthogonal polynomials over the canonical basis}. The linear independence of $\{x^n\}_{n\in\N}$ implies that we can equate the expressions of both sides of the recurrence relation. After equating the coefficients of $x^n$, we obtain
\[
\tau_{n+1,n}^\parameter{d}=\tau_{n,n-1}^\parameter{d}-\beta_n^\parameter{d},
\]
which implies
\[\beta_n^\parameter{d}=\tau_{n,n-1}^\parameter{d}-\tau_{n+1,n}^\parameter{d}. 
\]
Similarly, a comparison of the coefficients of $x^{n-1}$ in  \eqref{recurrence relation satisfied by the 2-orthogonal polynomials} brings the identity 
\[
\tau_{n+1,n-1}^\parameter{d}=\tau_{n,n-2}^\parameter{d}-\beta_n^\parameter{d}\tau_{n,n-1}^\parameter{d}-\alpha_n^\parameter{d},
\]
which can be rearranged to give 
\[\alpha_n^\parameter{d}=\tau_{n,n-2}^\parameter{d}-\tau_{n+1,n-1}^\parameter{d}-\left(\tau_{n,n-1}^\parameter{d}\right)^2+\tau_{n,n-1}^\parameter{d}\tau_{n+1,n}^\parameter{d}.
\]
Based on the expression for these $\tau$-coefficients in \eqref{expansion of the $2$-orthogonal polynomials over the canonical basis}, we have 
\begin{equation*}
\tau_{n,n-1}^\parameter{d}(a,b,c)=
-\frac{n(a+n)(b+n)}{c+b+\floor{\frac{n+d}{2}}+n}
\text{ and }
\tau_{n,n-2}^\parameter{d}(a,b,c)=
\frac{1}{2}\frac{n(a+n)(b+n)(n-1)(a+n-1)(b+n-1)}
{\left(c+b+\floor{\frac{n+d}{2}}+n\right)\left(c+b+\floor{\frac{n+d}{2}}+n-1\right)},  
\end{equation*}
which leads to 
\begin{subequations}
\begin{align}
&\beta_{2m+d}^\parameter{d}(a,b,c)\notag\\
&\qquad =\frac{(2m+d+1)(a+2m+d+1)(b+2m+d+1)}{c+b+3m+2d+1}-\frac{(2m+d)(a+2m+d)(b+2m+d)}{c+b+3m+2d}, 
		\label{betas 1}	\\[0.4cm]
&\beta_{2m+d}^\parameter{1-d}(a,b,c) \notag\\
&\qquad
=\frac{(2m+d+1)(a+2m+d+1)(b+2m+d+1)}{c+b+3m+d+2}-\frac{(2m+d)(a+2m+d)(b+2m+d)}{c+b+3m+d}
	\label{betas 2}
\end{align}
as well as  
\begin{align}
&\alpha_{2m+d}^\parameter{1-d}(a,b,c+d) 
=\alpha_{2m+d}^\parameter{d}(a,b,c) 	
	\label{alphas 1}\\[0.4cm]
&
\qquad=\frac{(2m+d)(a+2m+d)(b+2m+d)}{c+b+3m+2d}
\bigg(\frac{(2m+d-1)(a+2m+d-1)(b+2m+d-1)}{2(c+b+3m+2d-1)} 
	\label{alphas 2}\\
\notag&
\qquad-\frac{(2m+d)(a+2m+d)(b+2m+d)}{c+b+3m+2d} 
+\frac{(2m+d+1)(a+2m+d+1)(b+2m+d+1)}{2(c+b+3m+2d+1)}\bigg).
\end{align}
%

The expressions for the coefficients $\gamma_n^\parameter{d}$ could also be obtained in an analogous way after comparing the coefficients of $x^{n-2}$ in \eqref{recurrence relation satisfied by the 2-orthogonal polynomials}. However, it is rather  easier from the computational point of view, to derive such expressions directly from the $2$-orthogonality conditions. Indeed, the $2$-orthogonality conditions applied to the recurrence relation \eqref{recurrence relation satisfied by the 2-orthogonal polynomials}, straightforwardly imply that 
\begin{align*}
\gamma_{2n+1}^\parameter{d}(a,b,c)
=\frac{\int_{0}^{\infty}x^{n+1}P_{2n+2}^\parameter{d}(x;a,b,c)\W(x;a,b,c+d)\dx}
{\int_{0}^{\infty}x^nP_{2n}^\parameter{d}(x;a,b,c)\W(x;a,b,c+d)\dx}
\end{align*}
and
\begin{align*}
\gamma_{2n+2}^\parameter{d}(a,b,c)
=\frac{\int_{0}^{\infty}x^{n+1}P_{2n+3}^\parameter{d}(x;a,b,c)\W(x;a,b,c+1-d)\dx}
{\int_{0}^{\infty}x^nP_{2n+1}^\parameter{d}(x;a,b,c)\W(x;a,b,c+1-d)\dx}.
\end{align*}
Based on the latter alongside with \eqref{ortho Pn 0}-\eqref{N2k plus 1}, we deduce 
\begin{align}
    &\begin{multlined}\gamma_{2m+d}^\parameter{d}(a,b,c)\\
    =\frac{\pochhammer[2]{2m+d}\pochhammer[2]{a+2m+d}\pochhammer[2]{b+2m+d}(c+m+d)(c+b-a+m+d)(c+b+m+d)}
    {\pochhammer[3]{c+b+3m-1+2d}\pochhammer[3]{c+b+3m+2d}}, 
    \end{multlined}\\[0.4cm]
    &
    \gamma_{2m+d}^\parameter{1-d}(a,b,c)
    =\frac{\pochhammer[2]{2m+d}\pochhammer[2]{a+2m+d}\pochhammer[2]{b+2m+d}}{\pochhammer[3]{c+b+3m+d}}
    	\label{gammas 2}. 
\end{align}
\end{subequations}

As a consequence, we have just proved the following result. 

\begin{theorem}
\label{recurrence relation satisfied by the 2-OPS and asymptotic behaviour of the recurrence coefficients}
Let $a,b,c\in\R$ such that $a>-1$, $b>-1$ and $c>\max\{0,a-b\}$ and let $\dis\polyseq[P_n^\parameter{d}(x):=P_n^\parameter{d}(x;a,b,c)]$, $d\in\{0,1\}$, be the monic $2$-orthogonal polynomial sequence with respect to $\dis\Wvec^\parameter{d}(x;a,b,c)$.
Then the recurrence relation \eqref{recurrence relation satisfied by the 2-orthogonal polynomials} holds and the recurrence coefficients are given by \eqref{betas 1}-\eqref{gammas 2} with $\gamma_{m}^\parameter{d}(a,b,c)>0$ for all $m\geq1$. Furthermore, they have the following periodic asymptotic behaviour of period $2$: 
\begin{align*}
&\beta_{2m+d}^\parameter{d}(a,b,c)
\sim\frac{28}{9}\,m
	,\quad 
\beta_{2m+d}^\parameter{1-d}(a,b,c), 
\sim\frac{20}{9}\,m, \\[0.4cm]
&
\alpha_{2m+d}^\parameter{1-d}(a,b,c+d) =\alpha_{2m+d}^\parameter{d}(a,b,c) 
\sim\frac{208}{81}\,m^2,
	\\[0.4cm]
& \gamma_{2m+d}^\parameter{d}(a,b,c)
\sim\frac{2^6}{3^6}\,m^3,
	\quad\text{and}\quad 
\gamma_{2m+d}^\parameter{1-d}(a,b,c)
\sim\frac{2^6}{3^3}\,m^3,
 \quad 
\text{as} \quad m\to\infty.
\end{align*}
\end{theorem}

\subsection{Asymptotic behaviour of the largest zero}
We have already stated that, because $\W(x;a,b,c)$ and $\W(x;a,b,c+1)$ form a Nikishin system, then, if $\dis\polyseq$ is the $2$-orthogonal polynomial sequence with respect to these weight functions, $P_n$ has $n$ real positive simple zeros and the zeros of consecutive polynomials interlace as there is always a zero of $P_n$ between two consecutive zeros of $P_{n+1}$.
An asymptotic upper bound for the largest zero of $2$-orthogonal polynomial sequences is intimately related to the asymptotic behaviour of their recurrence relation coefficients, as explained in the following theorem.

\begin{theorem}
\label{asymptotic behaviour of the largest zero - more general theorem}
Suppose that $\polyseq$ is a $2$-orthogonal polynomial sequence satisfying \eqref{recurrence relation for a 2-OPS} 
with $\gamma_n>0$, for all $n\in\N$, and 
$\dis\gamma_n\leq\gamma\left(n^{3\lambda}+o(n^{3\lambda})\right)$, 
$\dis\alpha_n\leq\alpha\left(n^{2\lambda}+o(n^{2\lambda})\right)$ 
and $\dis\beta_n\leq\beta\left(n^{\lambda}+o(n^{\lambda})\right)$, 
\vspace*{0,1 cm}
for some real constants $\gamma>0$ and $\alpha,\beta,\lambda\geq 0$ such that $\dis\Delta:=\gamma^2-\frac{\alpha^3}{27}>0$.
\vspace*{0,1 cm}
Then, if we denote by $x_n^\roundparameter{n}$ the largest zero in absolute value of $P_n(x)$,
$\dis\left|x_n^\roundparameter{n}\right|
\leq\left(\frac{3}{2}\tau+\beta+\frac{\alpha}{2\tau}\right)n^{\lambda}
+o\left(n^{\lambda}\right)$, $n\to+\infty$,
where $\dis\tau=\sqrt[3]{\gamma+\sqrt{\Delta}}+\sqrt[3]{\gamma-\sqrt{\Delta}}$.
\end{theorem}

Recalling the asymptotic behaviour for the recurrence coefficients obtained in Theorem \ref{recurrence relation satisfied by the 2-OPS and asymptotic behaviour of the recurrence coefficients}, finding an asymptotic upper bound for the largest zero of $P_n^\parameter{d}(x;a,b,c)$ is an immediate consequence of Theorem \ref{asymptotic behaviour of the largest zero - more general theorem}. \\

\begin{corollary}
\label{asymptotic behaviour of the largest zero - our case}
Let $a,b,c\in\R$ such that $a>-1$, $b>-1$ and $c>\max\{0,a-b\}$ and let $\dis\polyseq[P_n^\parameter{d}(x;a,b,c)]$, $d\in\{0,1\}$, be the monic $2$-orthogonal polynomial sequence with respect to $\dis\Wvec^\parameter{d}(x;a,b,c)$.
Then $P_n^\parameter{d}$ has $n$ simple real positive zeros and, if we denote by $x_n^\roundparameter{n}$ the largest zero of $P_n^\parameter{d}(x;a,b,c)$, then 
\[ 
	x_n^\roundparameter{n}<M\cdot n+o(n), \quad \text{as}\quad n\to+\infty, 
\]
where $\dis M=\frac{3}{2}\tau+\beta+\frac{\alpha}{2\tau}\approx 3.484$,
with  $\dis\alpha=\frac{52}{81}$, $\dis\beta=\frac{14}{9}$, $\dis\gamma=\frac{8}{27}$, $\dis\Delta=\gamma^2-\frac{\alpha^3}{27}=\frac{1119104}{14348907}>0$ and $\dis\tau=\sqrt[3]{\gamma+\sqrt{\Delta}}+\sqrt[3]{\gamma-\sqrt{\Delta}}$.
\end{corollary}

We illustrate the latter result in Figure \ref{Fig Zeros}, produced in {\it Maple}. The curve $y=3.484 n$ gives clearly an upper bound for the largest zero of $P_n^\parameter{d}(x;a,b,c)$ for each $d\in\{0,1\}$. As already explained, the even order polynomials do not depend on $d$. Therefore, the zeros of $P_{2n}^\parameter{0}$ and $P_{2n}^\parameter{1}$ coincide, but a similar remark does not apply for the odd order polynomials because $P_{2n+1}^\parameter{0}\neq P_{2n+1}^\parameter{1}$. A sharper upper bound could be obtained if we consider further terms in the estimation and adapting the proof accordingly. For the purpose of this investigation this is not so relevant. 

\begin{figure}[htb]
\includegraphics[scale=0.8]{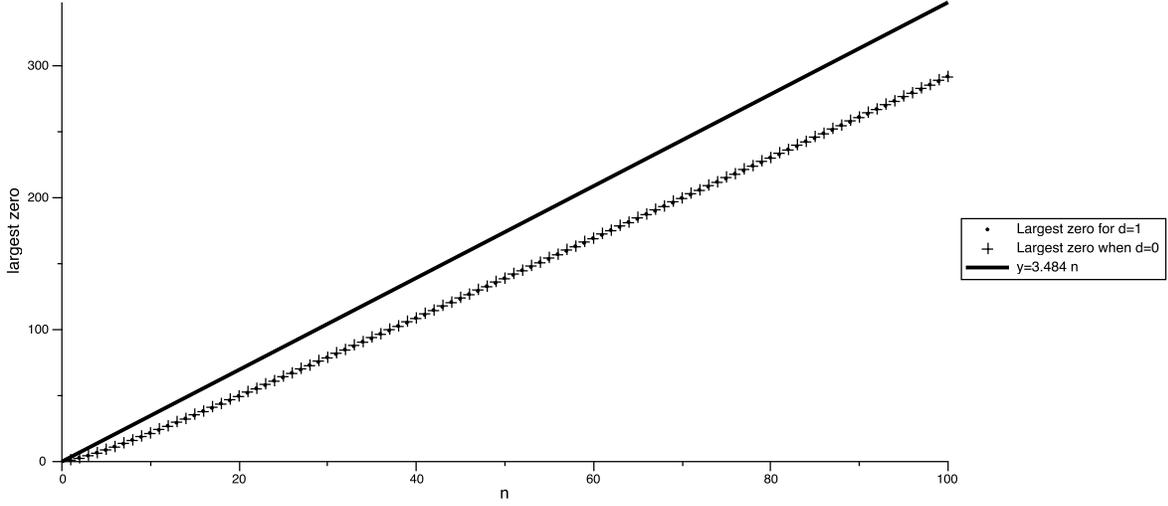}
\caption{Joint plots of the largest zeros of $P_n^\parameter{d}(x;2,1.5,5)$ for $d=0$ (crosses) and $d=1$ (dots)  for each $n=1,\ldots,100$ with the the upper bound curve $y=3.484 x$ in solid line. }
\label{Fig Zeros}
\end{figure}

Observe that Theorem \ref{asymptotic behaviour of the largest zero - more general theorem} is a generalisation of \cite[Th.~2.2]{AnaandWalter}, which is obtained from it when we set $\alpha=\beta=0$. The proof presented below is inspired on the proof of \cite[Th.~2.2]{AnaandWalter}.

\begin{proof}[Proof of Theorem \ref{asymptotic behaviour of the largest zero - more general theorem}.]
Consider the Hessenberg matrix
\begin{align*}
\Hn=\begin{bmatrix} 
 \beta_0 & 1 & 0 & 0 & \cdots & 0 \\
\alpha_1 & \beta_1 & 1 & 0 & \cdots & 0 \\ 
\gamma_1 & \alpha_2 & \beta_2 & 1 & \ddots & \vdots \\
 	   0 & \ddots & \ddots & \ddots & \ddots & 0 \\
  \vdots & \ddots & \gamma_{n-3} & \alpha_{n-2} & \beta_{n-2} & 1\\
 	   0 & \cdots & 0 & \gamma_{n-2} & \alpha_{n-1} & \beta_{n-1}
\end{bmatrix}
\end{align*}
so that the recurrence relation can be expressed as 
\begin{align*}
\Hn\begin{bmatrix} P_0(x) \\ P_1(x) \\  \vdots \\ P_{n-2}(x) \\ P_{n-1}(x) \end{bmatrix}
= x\begin{bmatrix} P_0(x) \\ P_1(x) \\  \vdots \\ P_{n-2}(x) \\ P_{n-1}(x) \end{bmatrix}
- P_n(x) \begin{bmatrix} 0 \\ 0 \\  \vdots \\ 0 \\ 1 \end{bmatrix}
\end{align*}
and each zero of $P_n$ is an eigenvalue of the matrix $\Hn$.
Then, if $\dis\rho(\Hn)=\max\{|\lambda|:\lambda\text{ is an eigenvalue of }\Hn\}$
is the spectral radius of the matrix $\Hn$, $\left|x_n^\roundparameter{n}\right|<\rho(\Hn)$. 

Moreover, $\rho(\Hn)$ is bounded from above by the matrix norm (see \cite[Section 5.6]{MatrixAnalysis})
\begin{align*}
\|\Hn\|_S=\|S^{-1}\Hn S\|_\infty=
\max_{1\leq i\leq n}\left\{\sum_{j=1}^{n}|(S^{-1}\Hn S)_{i,j}|\right\},
\end{align*}
for a non-singular $n\times n$ matrix $S$.
In particular, if we set $S=\diag(s_1,\cdots,s_n)$, with $\det(S)=s_1\cdots s_n\neq 0$, then
\begin{align*}
\|\Hn\|_S=
\max_{1\leq i\leq n}
\bigg\{ &
\left|\frac{s_2+\beta_0s_1}{s_1}\right|,
\left|\frac{s_3+\beta_1s_2+\alpha_1s_1}{s_2}\right|
\left|\frac{s_4+\beta_2s_3+\alpha_2s_2+\gamma_1s_1}{s_3}\right|,
\cdots, \\&
\left|\frac{s_n+\beta_{n-2}s_{n-1}+\alpha_{n-2}s_{n-2}+\gamma_{n-3}s_{n-3}}{s_{n-1}}\right|,
\left|\frac{\beta_{n-1}s_n+\alpha_{n-1}s_{n-1}+\gamma_{n-2}s_{n-2}}{s_n}\right|
\bigg\}.
\end{align*}

If we take $s_k=s^k(k!)^{\lambda}$, for some $s>0$, then, using the asymptotic behaviour of their recurrence coefficients in \eqref{recurrence relation satisfied by the 2-orthogonal polynomials}, we have
\begin{align*}
\left|\frac{s_k+\beta_{k-2}s_{k-1}+\alpha_{k-2}s_{k-2}+\gamma_{k-3}s_{k-3}}{s_{k-1}}\right|
\leq\left(s+\beta+\frac{\alpha}{s}+\frac{\gamma}{s^2}\right)k^{\lambda}
+o\left(k^{\lambda}\right),
\;k\to+\infty,
\end{align*}
and, as a result,
\begin{align*}
\|\Hn\|_S
\leq\left(s+\beta+\frac{\alpha}{s}+\frac{\gamma}{s^2}\right)n^{\lambda}
+o\left(n^{\lambda}\right),
\;n\to+\infty.
\end{align*}

To find the sharpest upper bound for $\dis\|\Hn\|_S$ \Big(and, as a consequence, for $\dis\left|x_n^\roundparameter{n}\right|$\Big) given by this formula we need to find the minimum value of $\dis f(s)=s+\beta+\frac{\alpha}{s}+\frac{\gamma}{s^2}$ on $\R^+$.
With that purpose, we look for the roots of
$\dis f'(s)=1-\frac{\alpha}{s^2}-\frac{2\gamma}{s^3} =\frac{1}{s^3}\left(s^3-\alpha s-2\gamma\right)$.
Due to the condition $\Delta>0$, we know that $f'$ has one real root and two complex roots.
\vspace*{0,1 cm}
Moreover, the real root is 
$\dis\tau=\sqrt[3]{\gamma+\sqrt{\Delta}}+\sqrt[3]{\gamma-\sqrt{\Delta}}>0$
(where we are taking real and positive square and cubic roots).
Furthermore, $\dis f''(s)=\frac{2\alpha}{s^3}+\frac{6\gamma}{s^4}$ hence $\dis f''(\tau)>0$ and, consequently, the choice $s=\tau$ gives a minimum value to $\dis f(s)=s+\beta+\frac{\alpha}{s}+\frac{\gamma}{s^2}$.
Finally, $f'(\tau)=0$ implies $\dis\frac{2\gamma}{\tau^3}=1-\frac{\alpha}{\tau}$ therefore 
$\dis f(\tau)=\frac{3}{2}\,\tau+\beta+\frac{\alpha}{2\tau}$, which implies the result. 
\end{proof}

\subsection{Confluence relation with the modified Bessel weights}
\label{Confluence relation with the modified Bessel weights}

There is a clear relation by confluence of these multiple orthogonal polynomials to those studied independently in  \cite{BenCheikhandDouak} and  \cite{SemyonandWalter}. In the latter, the study addressed multiple orthogonal polynomials of both types, while the former concentrated on the type II.  In both works, the focus was on weights involving the  modified Bessel functions and defined on the positive real line, for parameters $a,b>-1$, as follows
\begin{align}
\label{Bessel weights definition}
\V_0(x;a,b)=\frac{2}{\Gamma(a+1)\Gamma(b+1)}x^\frac{a+b}{2}\BesselK[2\sqrt{x}]{a-b}
\text{ and }
\V_1(x;a,b)=-\DiffOp\left(\V_0(x;a+1,b+1)\right),
\end{align}
where, as mentioned in the introduction, $\BesselK{\nu}$ is the modified Bessel function of second kind (also known as Macdonald function).

Let $\dis\polyseq[R_n(x;a,b)]$ be the $2$-orthogonal polynomial sequence with respect to the weight functions $\V_0(x;a,b)$ and $\V_1(x;a,b)$ supported on the positive real line. 
Similar to the $2$-orthogonal polynomials with respect to the modified Tricomi weights, this sequence can also be explicitly represented as a sequence of generalised hypergeometric polynomials by
\begin{align}
\label{2-orthogonal polynomials with respect to the modified Bessel weights}
R_n(x;a,b)=(-1)^n\pochhammer{a+1}\pochhammer{b+1}\Hypergeometric{1}{2}{-n}{a+1,b+1}. 
\end{align}

As a direct consequence of this representation, combined with the differential formula for the generalised hypergeometric series (see \cite[Eq.~16.3.1]{DLMF}), the sequence of derivatives of $\dis\polyseq[R_n(x;a,b)]$ is also $2$-orthogonal and it corresponds to the same sequence with shifted parameters.
More precisely, for any $n\in\N$,
\begin{align}
\label{differentiation formula for type II polynomials with respect to the modified Bessel weights}
\DiffOp\left(R_{n+1}(x;a,b)\right)=(n+1)R_n(x;a+1,b+1).
\end{align}
So, it fits within the family of Hahn-classical $2$-orthogonal polynomials. 

To obtain a limiting relation between the $2$-orthogonal polynomials with respect to the modified Bessel and Tricomi weights, recall that, if $p\leq q$, then the generalised hypergeometric series ${}_{p+1}F_q$ satisfies the confluent relation (see \cite[Eq.~16.8.10]{DLMF})
\begin{equation}
\lim_{|\alpha|\to\infty}
\Hypergeometric[\frac{x}{\alpha}]{p+1}{q}{\alpha_1,\cdots,\alpha_p,\alpha}{\beta_1,\cdots,\beta_q}
=\Hypergeometric{p}{q}{\alpha_1,\cdots,\alpha_p}{\beta_1,\cdots,\beta_q}.
\end{equation}

Using this formula to compare \eqref{explicit formula for the 2-orthogonal polynomials as a 2F2} and \eqref{2-orthogonal polynomials with respect to the modified Bessel weights}, we obtain the confluent relation
\begin{equation}
\label{confluent relation between the modified Tricomi and Bessel weights}
\lim_{c\to\infty}P_n^\parameter{d}\left(\frac{x}{c};a,b,c\right)=R_n\left(x;a,b\right).
\end{equation}

As expected, we can also obtain an equivalent confluent relation for the weight functions $\V_0(x;a,b)$ and $\W(x;a,b,c)$ as
\begin{equation*}
\lim_{c\to\infty}\frac{1}{c}\,\W\left(\frac{x}{c};a,b,c\right)=\V_0\left(x;a,b\right),
\end{equation*}
which is obtained  after taking $\nu=b-a$ in the following limiting relation between the modified Bessel function and the Tricomi function (see \cite[p.266]{BatemanProjectVol1})
\begin{equation*}
\lim_{c\to\infty}\Gamma(c+\nu)\KummerU[\frac{x}{c}]{a}{1-\nu}=2x^\frac{\nu}{2}\BesselK{\nu}.
\end{equation*}

\section{Connection with Hahn-classical $3$-fold symmetric polynomials}
\label{Connection with Hahn-classical 3-fold symmetric 2-orthogonal polynomials}

There are four distinct families of Hahn-classical threefold symmetric $2$-orthogonal polynomials, up to a linear transformation of the variable. A fact that was highlighted in  \cite{DouakandMaroniClassiquesDeDimensionDeux} and all these families were studied in detail in \cite{AnaandWalter}. The four arising cases were therein denominated as A, B1, B2 and C. The simplest is case A, which consists of $2$-orthogonal Appell polynomials, with no parameter dependence, and whose cubic components are particular cases of the $2$-orthogonal polynomials mentioned in Section \ref{Confluence relation with the modified Bessel weights} involving  modified Bessel weights and previously studied in \cite{SemyonandWalter} and \cite{BenCheikhandDouak}. The cases B1 and B2 have a richer structure, depend on a parameter and are related to each other via differentiation  articulated with parameter shift.  Their three cubic components are particular cases of the $2$-orthogonal polynomials studied in Section \ref{Type II multiple orthogonal polynomials}. In other words, particular choices on the parameters $(a,b,c,d)$ in \eqref{explicit formula for the 2-orthogonal polynomials as a 2F2} allows to describe the three cubic components $P_n^{[k]}$ as in \eqref{cubic comp} for these two cases. More precisely, for each of the cubic components indexed with $k\in\{0,1,2\}$ we have 
\[ P_n^\kk(x;\mu)=P_n^\parameter{d_k}\left(x;a_k,b_k,\frac{\mu}{3}\right) \quad \text{in case B1}\]
and	
\[ P_n^\kk(x;\rho)=P_n^\parameter{1-d_k}\left(x;a_k,b_k,\frac{\rho-2}{3}+d_k\right) \quad \text{in case B2}, \]
where 
\begin{equation}\label{ak bk dk}
 \left(a_0,b_0\right)=\left(-\tfrac{1}{3},-\tfrac{2}{3}\right), \  
 \left(a_1,b_1\right)=\left(-\tfrac{1}{3},\tfrac{1}{3}\right),  
 \left(a_2,b_2\right)=\left(\tfrac{2}{3},\tfrac{1}{3}\right)
\ \text{and}\  
  d_k = a_k-b_k+\tfrac{1}{3} =\tfrac{1-(-1)^k}{2}.
\end{equation}

As reported above, the cubic components $\dis P_n^\parameter{k}(x)$ for Case A are obtained from particular choices on the parameters of the $2$-orthogonal polynomials $\dis\polyseq[R_n(x;a,b)]$ in \eqref{2-orthogonal polynomials with respect to the modified Bessel weights}. Precisely, for each $k\in\{0,1,2\}$ we have $P_n^\parameter{k}(x)=R_n\left(\frac{x}{9};a_k,b_k\right)$ with  $(a_k,b_k)$ as in \eqref{ak bk dk}.  As expected, the confluent relation \eqref{confluent relation between the modified Tricomi and Bessel weights} generalises a limiting relation observed in \cite{DouakandMaroniClassiquesDeDimensionDeux} for Hahn-classical $3$-fold symmetric $2$-orthogonal polynomials: by taking $\mu,\rho\to\infty$ in cases B1 and B2, respectively, leads to case A.

Another observation lies on the fact that the cubic decomposition of threefold symmetric $2$-orthogonal polynomials preserves the Hahn-classical property. It is certainly true for cases A and B, if we take into account the identities \eqref{differentiation formula for type II polynomials with respect to the modified Bessel weights} and \eqref{differentiation formula for type II polynomials}, respectively. This property is rather intrinsic to all threefold symmetric $2$-orthogonal polynomials Hahn-classical polynomials, as we show below in Theorem  \ref{The cubic decomposition preserves the Hahn-classical property - theorem}. As a consequence, the cubic components in case C are also part of the Hahn-classical family. A further benefit from this result is on the techniques involved. Among other things, they can be adapted to prove analogous results regarding Hahn-classical polynomials with respect to other annihilating operators such as the $q$-derivative. \\

\begin{theorem}
\label{The cubic decomposition preserves the Hahn-classical property - theorem}
Let $\dis\polyseq$ be a $3$-fold symmetric Hahn-classical $2$-orthogonal polynomial sequence.
Then each of the three cubic components $\polyseq[P_n^\kk(x)]$, given by \eqref{cubic comp} with $k\in\{0,1,2\}$, is Hahn-classical. 
\end{theorem}

The proof of  Theorem \ref{The cubic decomposition preserves the Hahn-classical property - theorem} consists of showing that, under the assumptions, $\polyseq[\frac{1}{n+1}\DiffOp P_{n+1}^\kk(x)]$ is also $2$-orthogonal. To do so, we first need to derive the orthogonality weights for the cubic components of a $3$-fold symmetric Hahn-classical $2$-orthogonal polynomial sequence, which are explained in Proposition \ref{integral representation of the measures of 2-orthogonality of the cubic components}. For that purpose, we recall two auxiliary results obtained in \cite{AnaandWalter} and \cite{MaroniOrthogonalite}, respectively. 
The first providing the structure of the $2$-orthogonality measures for threefold symmetric Hahn-classical polynomials. The second to give the structure for the measures, written in terms of linear functionals, associated with the corresponding cubic components. 

\begin{proposition}
\label{Integral representation of a 3-fold symmetric Hahn-classical 2-OPS - proposition}
(cf. \cite[Theorem 3.3]{AnaandWalter})
Let $\dis\polyseq$ be a $3$-fold symmetric Hahn-classical $2$-orthogonal polynomial sequence satisfying \eqref{recurrence relation for a 2-OPS} with $\beta_n=\alpha_n=0$ and $\gamma_{n+1}>0$, for all $n\in\N$. 
Then $\polyseq$ is $2$-orthogonal with respect to a pair of measures $\left(\mu_0,\mu_1\right)$ admitting the integral representations
\begin{align}
\label{Integral representation over the 3-star of a 3-fold symmetric Hahn-classical 2-OPS}
\int_{S_3}f(x)\dd\mu_j(x)
=\frac{1}{3}
\left(\int_{0}^{\gamma}f(x)\,\UCal_j(x)\dx
+\omega^{2-j}\int_{0}^{\gamma\omega}f(x)\,\UCal_j(\omega^2 x)\dx
+\omega^{j+1}\int_{0}^{\gamma\omega^2}f(x)\,\UCal_j(\omega x)\dx
\right),
\end{align}
where $j\in\{0,1\}$, 
$\dis\omega=\e^{\frac{2\pi i}{3}}$,
$\dis\gamma=\lim_{n\to\infty}\frac{27}{4}\,\gamma_n$, 
$\dis S_3=\bigcup\limits_{k=0}^{2}[0,\gamma\,\omega^k]$ and $\UCal_j:[0,\gamma]\to\R$ are two twice differentiable functions satisfying the matrix differential equation
\begin{align}
\label{original matrix differential equation}
\DiffOp\left(\Phi(x)\twovector{\UCal_0(x)}{\UCal_1(x)}\right)+\Psi(x)\twovector{\UCal_0(x)}{\UCal_1(x)}=0,
\end{align}
where
\begin{align}
\label{matrices Phi and Psi in the original matrix differential equation}
\Phi(x)=
\matrixtwobytwo{\Theta_1}{\left(1-\Theta_1\right)x}
{\frac{2\left(1-\Theta_2\right)}{\gamma_1}\,x^2}{2\Theta_2-1}
\quad \text{and}\quad 
\Psi(x)=
\begin{bmatrix}
						0 & 1 \vspace*{0,1 cm} \\
\dis\frac{2}{\gamma_1}\,x & 0
\end{bmatrix},
\end{align}
for some constants $\Theta_1$ and $\Theta_2$ such that $\dis\Theta_1,\Theta_2\neq\frac{n-1}{n}$, for any $n\geq 1$.
\end{proposition}

In \cite{MaroniOrthogonalite} is shown that the cubic components of a threefold symmetric 2-orthogonal sequence are also $2$-orthogonal. The structure of the vector of two linear functionals for which the cubic components are $2$-orthogonal polynomial is also explained. We recall this result in Lemma \ref{linear functionals of 2-orthogonality of the cubic components}. Beforehand, and for a matter of completeness, we note that for any measure $\mu$ such that all moments exist and are finite, we can naturally define in $\PP'$ (the dual space of the vector space of polynomials $\PP$) a linear functional $u$ such that $\langle u,p\rangle=\int p(x)\dd\mu(x)$, for all $p\in\PP$. Given a polynomial sequence $\polyseq$ in $\PP$, we can build its corresponding dual sequence  $\polyseq[u_n]$ in $\PP'$ through $\dis\langle u_n,P_m\rangle=\delta_{nm}$. 

\begin{lemma} 
\label{linear functionals of 2-orthogonality of the cubic components}
\cite{MaroniOrthogonalite} (cf. \cite[Lemma 2.3.]{AnaandWalter})
Let $\dis\polyseq$ be a threefold symmetric $2$-orthogonal polynomial sequence with respect to a pair of linear functionals $\left(u_0,u_1\right)$ 
and let $\{u_n\}_{n\in\N}$ be the corresponding dual sequence. 
Then, for each $k\in\{0,1,2\}$, the cubic component $\dis\polyseq[P_n^\kk(x)]$ is $2$-orthogonal with respect to the vector of linear functionals $\left(u_0^\kk,u_1^\kk\right)$ such that  $u_0^\kk=\sigma_3(x^ku_k)$ and $u_1^\kk=\sigma_3(x^ku_{k+3})$, where $\sigma_3:\PP'\to\PP'$ represents the linear operator defined in $\PP'$ by $\langle\sigma_3(v),p(x)\rangle:=\langle v,f(x^3)\rangle$ for any $v\in\PP'$ and $p\in\PP$. 
\end{lemma}

As explained in \cite{MaroniRecursiveSets, MaroniSemiclassical}, all elements of the dual sequence of a $2$-orthogonal polynomial sequence can be written as a combination of the first two elements. Namely, there exists polynomials $E_n(x)$, $a_n(x)$, $F_n(x)$ and $b_n(x)$ such that $\deg E_n=\deg F_n=n$, $\deg \aaa_n\leq n$ and $\deg \bbb_n\leq n$ such that 
\begin{align*}
u_{2n}=E_n(x)u_0+\aaa_{n-1}(x)u_1
\text{ and }
u_{2n+1}=\bbb_n(x)u_0+F_n(x)u_1,\quad\text{for all} \quad n\in\N. 
\end{align*}
We have just used the product of a polynomial $f$ by a linear functional $u$, which is defined by duality: $\langle fu,p\rangle:=\langle u, f p\rangle$, for any $p\in\PP$.
The polynomials $E_n,\ F_n,\ \aaa_n$ and $\bbb_n$  satisfy recursive relations, which can be found in \cite{MaroniRecursiveSets}.  For our purpose, we need the expressions of the following ones: 
\begin{align}\label{u0 to u5}
\begin{cases} 
u_2= E_1u_0+\aaa_0u_1=\frac{x}{\gamma_1}u_0,
\\[0.2cm]
u_3=\bbb_1u_0+F_1u_1
=-\frac{1}{\gamma_2}u_0+\frac{x}{\gamma_2}u_1
=\frac{1}{\gamma_2}\left(xu_1-u_0\right),
\\[0.2cm]
u_4=E_2u_0+\aaa_1u_1
=\frac{x^2}{\gamma_1\gamma_3}u_0-\frac{1}{\gamma_3}u_1
=\frac{1}{\gamma_1\gamma_3}\left(x^2u_0-\gamma_1u_1\right),
\\[0.2cm]
u_5=\bbb_2u_0+F_2u_1
=-\left(\frac{1}{\gamma_1 \gamma_4}+\frac{1}{\gamma_2\gamma_4}\right)xu_0
+\frac{x^2}{\gamma_2 \gamma_4}u_1
=\frac{1}{\gamma_2\gamma_4}
\left(x^2u_1-\left(1+\frac{\gamma_2}{\gamma_1}\right)xu_0\right).
\end{cases}
\end{align}
The latter allows us to describe for each $k\in\{0,1,2\}$ the vector of functionals $\left(u_0^\kk,u_1^\kk\right)$ in Lemma \ref{linear functionals of 2-orthogonality of the cubic components}, which are used to derive the following result.  

\begin{proposition}
\label{integral representation of the measures of 2-orthogonality of the cubic components}
	
Suppose $\dis\polyseq$ is a $3$-fold symmetric polynomial sequence satisfying \eqref{recurrence relation for a 2-OPS}, with $\beta_n=\alpha_n=0$ and $\gamma_{n+1}>0$, whose $2$-orthogonality measures $\mu_0$ and $\mu_1$ admit the integral representations given by \eqref{Integral representation over the 3-star of a 3-fold symmetric Hahn-classical 2-OPS}. . 
Then the cubic components $\dis\polyseq[P_n^\kk(x)]$, $k\in\{0,1,2\}$, are $2$-orthogonal with respect to the pairs of measures $\left(\mu_0^\kk,\mu_1^\kk\right)$ admitting the integral representation
\begin{align}
\label{moments of the cubic decomposition sequences}
\int f(x)\dd\mu_j^\kk(x)=\int_{0}^{\gamma^3}f(x)\,\UCal_j^\kk(x)\dx,
\end{align}
where the weight functions $\UCal_j^\kk(x)$ are
\begin{subnumcases}{}
\label{cubic component orthognality weights - k=0}
\UCal_0^\zero(x)=
\frac{1}{3}\,x^{-\frac{2}{3}}\UCal_0\left(x^\frac{1}{3}\right)
\text{ and }
\UCal_1^\zero(x)=\frac{1}{3\gamma_2}
\left(x^{-\frac{1}{3}}\UCal_1\left(x^\frac{1}{3}\right)
-x^{-\frac{2}{3}}\UCal_0\left(x^\frac{1}{3}\right)\right), 
\\
\label{cubic component orthognality weights - k=1}
\UCal_0^\one(x)=\frac{1}{3}\, x^{-\frac{1}{3}}\UCal_1\left(x^\frac{1}{3}\right)
\text{ and }
\UCal_1^\one(x)=\frac{1}{3\gamma_1\gamma_3}
\left(x^{\frac{1}{3}}\,\UCal_0\left(x^\frac{1}{3}\right)
-\gamma_1x^{-\frac{1}{3}}\,\UCal_1\left(x^\frac{1}{3}\right)\right), 
\\
\label{cubic component orthognality weights - k=2}
\UCal_0^\two(x)
=\frac{1}{3\gamma_1}\,x^{\frac{1}{3}}\,\UCal_0\left(x^\frac{1}{3}\right)
\text{ and }
\UCal_1^\two(x)
=\frac{1}{3\gamma_2\gamma_4}
\left(x^{\frac{2}{3}}\,\UCal_1\left(x^\frac{1}{3}\right)
-\left(1+\frac{\gamma_2}{\gamma_1}\right)
x^{\frac{1}{3}}\,\UCal_0\left(x^\frac{1}{3}\right)\right).
\end{subnumcases}

\end{proposition}

\begin{proof}
Note that it is sufficient to prove \eqref{moments of the cubic decomposition sequences} for $f(x)=x^n$, for all $n\in\N$. Observe that the integral representations given by \eqref{Integral representation over the 3-star of a 3-fold symmetric Hahn-classical 2-OPS} imply that, for $j\in\{0,1\}$ and $n\in\N$, $\langle u_j, x^{3n+k}\rangle = 0 $, if $k\in\{0,1,2\}\backslash\{j\}$, and
\begin{align}
\label{integral representation of the nonzero moments}
\langle u_j , x^{3n+j} \rangle
=\int_{0}^{\gamma}x^{3n+j}\,\UCal_j(x)\dx
=\frac{1}{3}\int_{0}^{\gamma^3}t^{n+\frac{j-2}{3}}\,\UCal_j\left(t^\frac{1}{3}\right)\dd t.
\end{align}
Then, using Lemma \ref{linear functionals of 2-orthogonality of the cubic components}, we have
\begin{align}
\label{pf lemma eq1}
\int x^n\dd\mu_0^\zero(x)
=\langle \sigma_3(u_0),x^n\rangle
=\langle u_0,x^{3n}\rangle
=\frac{1}{3}\int_0^\gamma x^{n-\frac{2}{3}}\W_0\left(x^\frac{1}{3}\right)\dx
\end{align}
and
\begin{align}\label{pf lemma eq2}
\int x^n\dd\mu_0^\one(x)
=\langle \sigma_3(xu_1),x^n\rangle
=\langle xu_1,x^{3n}\rangle
=\langle u_1,x^{3n+1}\rangle
=\frac{1}{3}\int_0^\gamma x^{n-\frac{1}{3}}\UCal_1\left(x^\frac{1}{3}\right)\dx, 
\end{align}
which give the expressions for $\UCal_0^\zero(x)$ and  $\UCal_0^\one(x)$ in \eqref{cubic component orthognality weights - k=0} and \eqref{cubic component orthognality weights - k=1}, respectively.

Additionally, using the expressions for the elements of the dual sequence $u_2$, $u_3$, $u_4$ and $u_5$ given in \eqref{u0 to u5}, we successively get: 
\begin{align*}
& \int x^n\dd\mu_0^\two(x)
=\langle \sigma_3(x^2u_2),x^n\rangle
=\langle x^2u_2,x^{3n}\rangle
=\frac{1}{\gamma_1}\langle x^3u_0,x^{3n}\rangle
=\frac{1}{\gamma_1}\langle u_0,x^{3n+3}\rangle,
\\
& \int x^n\dd\mu_1^\zero(x)
=\langle \sigma_3(u_3),x^n\rangle
=\langle u_3,x^{3n}\rangle
=\frac{1}{\gamma_2}\left(\langle u_1,x^{3n+1}\rangle
-\langle u_0,x^{3n}\rangle\right),
\\
& \int x^n\dd\mu_1^\one(x)
=\langle \sigma_3(xu_4),x^n\rangle
=\langle xu_4,x^{3n}\rangle
=\frac{1}{\gamma_3}\left(\frac{1}{\gamma_1}\langle u_0,x^{3n+3}\rangle-\langle u_1,x^{3n+1}\rangle\right),
\\
& \int x^n\dd\mu_1^\two(x)
=\langle \sigma_3(x^2u_5),x^n\rangle
=\langle x^2u_5,x^{3n}\rangle
=\frac{1}{\gamma_2\gamma_4}
\left(\langle u_1,x^{3n+4}\rangle
-\left(1+\frac{\gamma_2}{\gamma_1}\right)
\langle u_0,x^{3n+3}\rangle\right), 
\end{align*}
which, because of the last identities in \eqref{pf lemma eq1}-\eqref{pf lemma eq2}, lead to
\begin{align*}
& \UCal_0^\two(x)=\frac{x}{\gamma_1}\,\UCal_0^\zero(x), 
\quad 
	\UCal_1^\zero(x)=\frac{1}{\gamma_2}\left(\UCal_0^\one(x)-\UCal_0^\zero(x)\right),
\quad 
	\UCal_1^\one(x)=\frac{1}{\gamma_3}\left(\UCal_0^\two(x)-\UCal_0^\one(x)\right), \\
& \UCal_1^\two(x)=\frac{x}{\gamma_2\gamma_4}
\left(\UCal_0^\one(x)-\left(1+\frac{\gamma_2}{\gamma_1}\right)\UCal_0^\zero(x)\right). 
\end{align*}
Finally, \eqref{cubic component orthognality weights - k=0}-\eqref{cubic component orthognality weights - k=2} follow directly from the latter identities, after we take into account the already obtained expressions for $\UCal_0^\zero(x)$ and $\UCal_0^\one(x)$. 
\end{proof}

We can now prove the main result in this section. 

\begin{proof}[Proof of Theorem  \ref{The cubic decomposition preserves the Hahn-classical property - theorem}.]
If $\dis\polyseq$ is a $3$-fold symmetric Hahn-classical $2$-orthogonal polynomial sequence, then the sequence of derivatives $\dis\polyseq[Q_n(x):=\frac{1}{n+1}\DiffOp \left(P_n(x)\right)]$ is also $3$-fold symmetric and $2$-orthogonal and, recalling Lemma \ref{linear functionals of 2-orthogonality of the cubic components}, the same holds for the cubic components $\dis\polyseq[Q_n^\kk(x)]$, $k\in\{0,1,2\}$.
As a result, it is straightforward to check that Theorem \ref{The cubic decomposition preserves the Hahn-classical property - theorem} is valid for $k=0$, that is, $\dis\polyseq[\frac{1}{n+1}\DiffOp P_{n+1}^\zero(x)]$ is a $2$-orthogonal polynomial sequence because
\begin{align}
\DiffOp\left(P_{n+1}^\zero(x)\right)
=\DiffOp\left(P_{3n+3}\left(x^{\frac{1}{3}}\right)\right)
=\frac{x^{-\frac{2}{3}}}{3}\,P'_{3n+3}\left(x^{\frac{1}{3}}\right)
=(n+1)x^{-\frac{2}{3}} Q_{3n+2}\left(x^{\frac{1}{3}}\right)
=(n+1)Q_{n}^\two(x).
\end{align} 
This observation was already made by Douak and Maroni in \cite{TchebyshevII}.
However, an analogous procedure does not give an obvious way to conclude about the $2$-orthogonality of $\dis\polyseq[\frac{1}{n+1}\DiffOp P_{n+1}^\kk(x)]$, for $k\in\{1,2\}$. So we take a different approach to prove this. 

According to Proposition \ref{differentiation formula for type II polynomials - general result}, to prove the Hahn-classical character of $\dis\polyseq[P_n^\kk(x)]$,  it is sufficient to find matrices
$\Phi^\kk(x)=\matrixtwobytwo{\phi_{00}}{\phi_{01}}{\varphi(x)}{\phi_{11}}$
and \vspace*{0,1 cm}
$\Psi^\kk(x)=\matrixtwobytwo{\eta_0}{\eta_1}{\psi(x)}{\xi}$,
for some constants $\phi_{00}$, $\phi_{01}$, $\phi_{11}$, $\eta_0$, $\eta_1$ and $\xi$ and polynomials $\varphi$ and $\psi$ with $\deg\varphi\leq 1$ and $\deg\psi=1$, such that
\begin{align}
\label{canonical differential equation W_k(x)}
\DiffOp\left(x\,\Phi^\kk(x)\UCalvec^\kk(x)\right)+\Psi^\kk(x)\UCalvec^\kk(x)=0,
\end{align}
where we use the notation $\dis\UCalvec^\kk(x):=\UCalvectorK{\kk}$.
Similarly, we also consider $\dis\UCalvec(x):=\UCalvector$.

To find the matrices $\dis\Phi^\kk(x)$ and $\dis\Psi^\kk(x)$, we start by rewriting formulas \eqref{cubic component orthognality weights - k=1} and \eqref{cubic component orthognality weights - k=2} as
\begin{align*}
\UCalvec^{\kk}(s)=\frac{1}{3}\,T_k(s)\UCalvec\left(s^{\frac{1}{3}}\right),
\end{align*}
where
\begin{align}
\nonumber
T_1(s)=
\matrixtwobytwo{0}{s^{-\frac{1}{3}}}
{\frac{1}{\gamma_1\gamma_3}\,s^{\frac{1}{3}}}
{-\frac{1}{\gamma_3}\,s^{-\frac{1}{3}}}
\text{ and }
T_2(s)=
\matrixtwobytwo{\frac{1}{\gamma_1}\,s^{\frac{1}{3}}}{0}
{-\frac{1}{\gamma_4}\left(\frac{1}{\gamma_2}+\frac{1}{\gamma_1}\right)s^{\frac{1}{3}}}
{\frac{1}{\gamma_2\gamma_4}\,s^{\frac{2}{3}}}.
\end{align}
These equations are naturally equivalent to
\begin{align}
\nonumber
\UCalvec\left(s^{\frac{1}{3}}\right)=3\,T_k^{-1}(s)\UCalvec^{\kk}(s),
\end{align}
with
\begin{align}
\nonumber
T_1^{-1}(s)
=\matrixtwobytwo{\gamma_1\,s^{-\frac{1}{3}}}
{\gamma_3\,\gamma_1\,s^{-\frac{1}{3}}}
{s^{\frac{1}{3}}}{0}
\text{ and }
T_2^{-1}(s)
=\matrixtwobytwo{\gamma_1\,s^{-\frac{1}{3}}}{0}
{\left(\gamma_1+\gamma_2\right)s^{-\frac{2}{3}}}
{\gamma_2\,\gamma_4\,s^{-\frac{2}{3}}}.
\end{align}

If we consider the change of variable $\dis s=x^\frac{1}{3}$ in matrix differential equation \eqref{original matrix differential equation} and then use the previous formula, we obtain, for both $k\in\{1,2\}$,
\begin{align} 
9\,x^\frac{2}{3}\DiffOp\left(\Phi\left(x^\frac{1}{3}\right)T_k^{-1}(x)\UCalvec^\kk(x)\right)
+3\Psi\left(x^\frac{1}{3}\right)T_k^{-1}(x)\UCalvec^\kk(x)=0,
\end{align}
or, equivalently,
\begin{align}
\label{matrix differential equation for the 2-orthogonality weights of the cubic components}
3\,x^\frac{2}{3}\Phi\left(x^\frac{1}{3}\right)T_k^{-1}(x)\DiffOp\left(\UCalvec^\kk(x)\right)
+\left(\Psi\left(x^\frac{1}{3}\right)T_k^{-1}(x)
+3x^\frac{2}{3}\DiffOp\left(\Phi\left(x^\frac{1}{3}\right)T_k^{-1}(x)\right)\right)\UCalvec^\kk(x)=0.
\end{align}

For $k=1$, relation \eqref{matrix differential equation for the 2-orthogonality weights of the cubic components} reads as
\begin{align*} 
\begin{multlined}
3\matrixtwobytwo{\left(1-\Theta_1\right)x^{\frac{4}{3}}+\Theta_1\gamma_1x^{\frac{1}{3}}}
{\Theta_1\gamma_1\gamma_3x^{\frac{1}{3}}}
{x}{6\left(1-\Theta_2\right)\gamma_3\,x}
\DiffOp\left(\UCalvec^\one(x)\right)
\\
+\matrixtwobytwo{\left(3-2\Theta_1\right)x^{\frac{1}{3}}-\Theta_1\gamma_1x^{-\frac{2}{3}}}
{-\Theta_1\gamma_1\gamma_3x^{-\frac{2}{3}}}
{3}{-2\left(2\Theta_2-1\right)\gamma_3}\UCalvec^\one(x)
=0, 
\end{multlined}
\end{align*}
which, after a  multiplication by $\dis\matrixtwobytwo{0}{1}{x^{\frac{2}{3}}}{0}$, leads to \eqref{canonical differential equation W_k(x)} for $k=1$, with
\begin{align*}
\Phi^\parameter{1}(x)=
3\matrixtwobytwo{\frac{1}{2\left(2\Theta_2-1\right)\gamma_3}}{\frac{1-\Theta_2}{2\Theta_2-1}}
{\left(1-\Theta_1\right)x+\Theta_1\gamma_1}{\Theta_1\gamma_1\gamma_3}
\text{ and }
\Psi^\parameter{1}(x)=\matrixtwobytwo{0}{1}
{\left(4\Theta_1-3\right)x-4\Theta_1\gamma_1}{-4\Theta_1\gamma_1\gamma_3}.
\end{align*}
As a result, $\dis\polyseq[P_n^\one(x)]$ is a Hahn-classical $2$-orthogonal polynomial sequence.

To prove the result for $k=2$, we start by using the relation $\dis\gamma_1=\frac{1}{3}\left(4\Theta_1-3\right)\gamma_2$ (see \cite[Theorem 3.2]{AnaandWalter} with $n=0$) to rewrite
\begin{align}
\nonumber
T_2^{-1}(s)
=\matrixtwobytwo
{\frac{1}{3}\left(4\Theta_1-3\right)\gamma_2\,s^{-\frac{1}{3}}}{0}
{\frac{4}{3}\,\Theta_1\gamma_2s^{-\frac{2}{3}}}
{\gamma_2\,\gamma_4\,s^{-\frac{2}{3}}}.
\end{align}
Therefore \eqref{matrix differential equation for the 2-orthogonality weights of the cubic components} becomes
\begin{align*} 
\begin{multlined}
3\matrixtwobytwo{\frac{1}{3}\Theta_1\gamma_2\,x^{\frac{1}{3}}}
{\gamma_2\gamma_4\left(1-\Theta_1\right)\,x^{\frac{1}{3}}}
{2\left(1-\Theta_2\right)x+\frac{4}{3}\gamma_2\left(2\Theta_2-1\right)}
{\gamma_2\gamma_4\left(2\Theta_2-1\right)}
\DiffOp\left(\UCalvec^\two(x)\right) \\
+\matrixtwobytwo{\Theta_1\gamma_2\,x^{-\frac{2}{3}}}
{\Theta_1\gamma_2\gamma_4\,x^{-\frac{2}{3}}}
{2\left(2-\Theta_2\right)+\frac{8}{3}\,\gamma_2\left(1-2\Theta_2\right)x^{-1}}
{-2\gamma_2\gamma_4\left(2\Theta_2-1\right)x^{-1}}
\UCalvec^\two(x)
=0, 
\end{multlined}
\end{align*}
which, after a multiplication by $\matrixtwobytwo{x^{\frac{2}{3}}}{0}{0}{x}$, corresponds to \eqref{canonical differential equation W_k(x)} for $k=2$, with
\begin{align*}
\Phi^\parameter{2}(x)
=\matrixtwobytwo{\frac{\Theta_1\gamma_2}{3\gamma_1\gamma_4}}
{\left(1-\Theta_1\right)\frac{\gamma_2}{\gamma_1}}
{\varphi^\parameter{2}(x)}{3\left(2\Theta_2-1\right)\gamma_2\gamma_4}
\quad \text{ and }\quad 
\Psi^\parameter{2}(x)
=\matrixtwobytwo{0}{1}
{\psi^\parameter{2}(x)}{-5\left(2\Theta_2-1\right)\gamma_2\gamma_4},
\end{align*}
where $\dis\varphi^\parameter{2}(x)=6\left(1-\Theta_2\right)x+4\left(2\Theta_2-1\right)\gamma_2$ and
$\dis\psi^\parameter{2}(x)=2\left(5\Theta_2-4\right)x-\frac{20}{3}\,\Theta_1\left(2\Theta_2-1\right)\gamma_2$.\\
Hence, $\dis\polyseq[P_n^\two(x)]$ is  Hahn-classical.
\end{proof}

We have shown here that the type II multiple orthogonal polynomials characterised in Section \ref{Type II multiple orthogonal polynomials} generalise the cubic components of cases B1 and B2 of the Hahn-classical $3$-fold symmetric $2$-orthogonal polynomials in a similar way to how the type II multiple orthogonal polynomials on the step line with respect to the modified Bessel weights, defined by \eqref{Bessel weights definition}, generalise the cubic components of case A.
As proved in Theorem \ref{The cubic decomposition preserves the Hahn-classical property - theorem}, the cubic components of case C are again Hahn-classical. It remains an open question if there is an analogous generalisation for these components and, in case there is one, if that generalisation can be such the differentiation gives a shift on parameters. In this scenario, we also expect the confluence relations between case C and cases B1 and B2 to be preserved. We defer this investigation to a forthcoming work. 

\thanks{{\bf Acknowledgements. } We would like to thank Guillermo Lop\'ez-Lagomasino, Abey Lop\'ez-Garc\'ia for ins\-tructive conversations on Nikishin systems and Walter Van Assche for  enlightening discussions regarding several aspects of this research. }

\bibliographystyle{plain}
\bibliography{MOPS}
\end{document}